\documentclass[12pt,oneside]{article}
\usepackage{amsmath,amssymb,amsfonts,amsthm}
 \pdfoutput=1
\usepackage{eurosym}
\usepackage{amsfonts}
\usepackage{amsmath}
\usepackage{geometry}
\usepackage{mathtools}
\usepackage{pdfpages}
\usepackage{amsthm}
\usepackage{indentfirst}
\usepackage{hyperref}
\usepackage{cleveref}
\crefformat{section}{\S#2#1#3} 
\crefformat{subsection}{\S#2#1#3}
\crefformat{subsubsection}{\S#2#1#3}
\hypersetup{
	colorlinks=true,
	linkcolor=blue,
	filecolor=magenta,
	urlcolor=cyan,
	citecolor=red}

\newtheorem{theorem}{Theorem}[section]
\newtheorem{remark}[theorem]{Remark}
\newtheorem{lemma}[theorem]{Lemma}
\newtheorem{corollary}[theorem]{Corollary}
\newtheorem{definition}[theorem]{Definition}
\newtheorem{example}[theorem]{Example}
\newtheorem{proposition}[theorem]{Proposition}

\numberwithin{equation}{section}

\newcommand{\C}{{\mathcal C}}

\newcommand{\quotes}[1]{``#1''}

\begin{document}
\title{\bf {Anisotropic conformal change of conic pseudo-Finsler surfaces, I}} 
 
 \author{{\bf  S. G. Elgendi$^{1}$,  Nabil L. Youssef$\,^2$,    A. A. Kotb$^3$ and Ebtsam H. Taha$^2$ }}
\date{}

\maketitle
\vspace{-1.15cm}

\begin{center}
 $^{1}$ Department of Mathematics, Faculty of Science,\\ Islamic University of Madinah,  Madinah, Kingdom of Saudia Arabia
\vspace{0.2cm}
\end{center}

\begin{center}
{$^2$Department of Mathematics, Faculty of Science,\\
Cairo University, Giza, Egypt}
\vspace{-8pt}
\end{center}

\begin{center}
{$^3$Department of Mathematics, Faculty of Science,\\
Damietta University, Damietta , Egypt}
\vspace{-8pt}
\end{center}

\begin{center}
salah.ali@fsc.bu.edu.eg, salahelgendi@yahoo.com\\
nlyoussef@sci.cu.edu.eg, nlyoussef2003@yahoo.fr\\
alikotb@du.edu.eg,  kotbmath2010@yahoo.com\\
ebtsam.taha@sci.cu.edu.eg, ebtsam.h.taha@hotmail.com

\end{center}

\vspace{0.7cm}

%%%%%%%%%%%%%%%%%%%%%%%%%%%%%%%%%%%%%%%%%%%%%%%%%%%%%%%%%%%%%%%%%%

\begin{abstract}
The present work is devoted to  investigate  anisotropic conformal transformation of  conic pseudo-Finsler surfaces $(M,F)$, that is, $ F(x,y)\longmapsto \overline{F}(x,y)=e^{\phi(x,y)}F(x,y)$, where the function $\phi(x,y)$ depends on both position $x$ and direction $y$, contrary to the ordinary (isotropic) conformal transformation which depends on position only. If $F$ is a pseudo-Finsler metric, the above transformation does not yield necessarily a pseudo-Finsler metric. Consequently, we find out necessary and sufficient condition for a (conic) pseudo-Finsler surface $(M,F)$ to be transformed to a (conic) pseudo-Finsler surface $(M,\overline{F})$ under the transformation $\overline{F}=e^{\phi(x,y)}F$. In general dimension, it is extremely difficult to find the anisotropic conformal change of the inverse metric tensor in a tensorial form. However, by using the modified Berwald frame on a Finsler surface, we obtain the change of the components of the inverse metric tensor in a tensorial form.  This progress enables us to study the transformation of the Finslerian geometric objects and the geometric properties associated with the transformed  Finsler function  $\overline{F}$.  In contrast to isotropic conformal transformation, we have a non-homothetic  conformal factor $\phi(x,y)$ that preserves the geodesic spray. Also, we find out some invariant geometric objects under the anisotropic conformal change.  Furthermore, we investigate a sufficient condition for $\overline{F}$ to be dually flat or/and  projectively flat.  Finally, we study some special cases of the conformal factor $\phi(x,y)$. Various examples are provided whenever the situation needs.
\end{abstract}



%%%%%%%%%%%%%%%%%%%%%%%%%%%%%%%%%%%%%%%%%%%%%%%%%%%%%%%%%%%%%%%%%%%%%
\section*{Introduction}
 Conformal transformations have been investigated in different frameworks, since they have some important geometric features \cite{ Hashiguchi76, Hassan,  Fuster, Shenbook16, Youssef18} and successful applications such as those in physics, biology and ecology \cite{ An.In.Ma1993,anisotropic22,  anisotropic, Maluf, Voicu}. 
 In the conformal transformation theory, it is useful to study the geometric properties and  geometric objects that are preserved under  conformal transformations. In Riemannian  geometry, (isotropic) conformal transformation is an angle-preserving. In pseudo Riemannian metrics with signature (Lorentzian metrics),  it preserves causality. In parallelizable manifolds, it preserves both angles and causality as each parallelization structure induces a pseudo-Riemannian (or Lorentzian) metric, cf. \cite{Maluf, Youssef18}.
In Finsler geometry, (isotropic) conformal transformation is angle-preserving \cite{Hashiguchi76} and leaves the geodesic spray  invariant if the conformal factor is a homothety \cite[Theorem 3.3]{Bacso Sandor-Cheng}. Further, each Randers space $(M, F= \alpha + \beta)$ has a globally defined nonholonomic frame which is called the Holland frame \cite{Holland}. This frame is an isotropic conformally invariant in the sense that the transformation $F(x,y)\longmapsto\overline{F}(x,y)=e^{\phi(x)}F(x,y)$ leaves the frame elements fixed \cite[Theorem 5.10.1]{Bucatarubook}.    

Given two pseudo-Finsler metrics $F(x,y)$ and $\overline{F}(x,y)$,  the anisotropic conformal transformation $F(x,y)\longmapsto \overline{F}(x,y)=e^{\phi(x,y)}F(x,y)$ preserves the lightcone \cite{anisotropic} and lightlike geodesics and their focal points \cite{ anisotropic2014, anisotropic}.  
The anisotropic conformal change in Finsler Geometry is not only valuable in applications, but also it has a great impact in the study of Finsler geometry. For example, in contrast to isotropic conformal change, the anisotropic conformal change of a pseudo-Finsler metric is not necessarily a pseudo-Finsler metric (see Theorem \ref{ness.suff.codtion.F.bar.Finsler}). In addition,  the anisotropic conformal change can send a psudo-Finsler metric to a pseudo-Riemannian one and vice-versa. 
 
Since conic pseudo-Finsler surfaces  are  used widely in applications, especially in physics, we study here the anisotropic conformal transformation $F(x,y)\longmapsto\overline{F}(x,y)=e^{\phi(x,y)}F(x,y)$ of a conic pseudo-Finsler metric $F(x,y)$.  First, under the anisotropic conformal change, we compute the components $\overline{g}_{ij}$ of the metric tensor of $\overline{F}$. Actually, the expression of $\overline{g}_{ij}$ contains the second derivative of the function $\phi$ with respect to directional arguments. Generally, this term  represents a severe obstacle to compute the components $\overline{g}^{ij}$ of the inverse metric tensor.
Keeping in mind that the inverse metric tensor is the door to find the geodesic spray, which enables to study the geometry of the transformed space $(M,\overline{F})$, this motivates us to consider the two-dimensional case to make use of the modified Berwald frame.  Consequently, we become able to derive some important geometric objects associated with the transformed metric $\overline{F}$, for example,  the transformed  Cartan  tensor,  main scalar, geodesic spray and  Barthel connection.  Further, we find out some invariant objects under anisotropic conformal transformations (see Theorem \ref{phi dh-closed}).

The present paper is organized in the following  manner.  
In~\S 1,  we recall some basic facts on the geometry of sprays and Finsler manifolds. 
In \S 2, we first show that the fact that  \quotes{two Finsler metrics are anisotropically conformally related} is not equivalent to the \quotes{proportionality of their associated  metric tensors}, contrary to what has been mentioned in  \cite[(11.1.1)]{Shenbook16}. Then, we investigate the anisotropic conformal transformations of  pseudo-Finsler surfaces. 
In \S 3, we characterize the anisotropic conformal transformations and  study their action on some important Finslerian geometric objects. Some anisotropic conformal invariant geometric properties are found out. 
In \S 4,  we study the case when the anisotropic conformal transformation is a projective transformation.  Moreover, we obtain sufficient conditions for $\overline{F}$ to be projectively flat and dually flat.   
Finally,  in \S 5, we consider two important special cases: the factor $\phi$ is a function of position only and the factor $\phi$ is a function of direction only. Interesting results in both cases are obtained. 

It should be noted that, various examples have been provided whenever the situation needs. The Maple’s code of these examples is presented in the Appendix at the end of the paper.

%%%%%%%%%%%%%%%%%%%%%%%%%%% Sec. 1 %%%%%%%%%%%%%%%%%%%%%%%%%%%%%%%%%

\section{Notation and preliminaries}
Let $ M $ be an n-dimensional smooth  manifold and $ \pi : TM \longrightarrow M $ the canonical projection of the tangent bundle $TM$ onto $M$. Let $TM_{0}:=TM\setminus(0)$ be the slit tangent bundle, $(0)$ being the null section, and $(x^{i}, y^i)$   the corresponding  coordinate system $TM$.  
  The natural almost-tangent structure $J$ of $T M$ is the vector $1$-form   defined  by $J = \frac{\partial}{\partial y^i} \otimes dx^i$. The vertical vector field
$\C=y^i\frac{\partial}{\partial y^i}$ on $TM$ is called  the
Liouville vector field. The set of smooth functions on $TM_0$ is denoted by $C^{\infty} (TM_0)$.

A spray on $M$ is a vector field $S$ on $ T M$ such that $JS = \C$ and
$[\C, S] = S$. Locally, it is  given by \cite{Gifone}
\begin{equation*}
  \label{eq:spray}
  S = y^i \frac{\partial}{\partial x^i} - 2G^i (x,y) \,\frac{\partial}{\partial y^i},
\end{equation*}
where the {spray coefficients} $G^i(x,y)$ are positively $2$-homogeneous
functions in  $y$ (or simply $h(2)$-functions). A nonlinear connection is defined by an $n$-dimensional distribution $H(TM_0) $ on $TM_0$  which is supplementary to the vertical distribution $V(TM_0) $. This means that for all $u \in TM_0$, we have
\begin{equation*}
%\label{eq:direct_sum}
T_u(TM_0) = V_u(TM_0) \oplus H_u(TM_0) .
\end{equation*}
The local basis of $V_u(TM_0)$ and $H_u(TM_0)$ are given, respectively by,  $\dot{\partial}_i:=\dfrac{\partial}{\partial y^i}$   and $\delta_i:=\dfrac{\delta}{\delta x^i}=\dfrac{\partial}{\partial x^i}-G^j_i(x,y)\dfrac{\partial}{\partial y^j}$. The coefficients of  Cartan nonlinear (or Barthel) connection  and Berwald connection are defined by  $G^j_i(x,y):=\dfrac{\partial G^j}{\partial y^i}$ and  $ G^h_{ij}(x,y):=\dfrac{\partial G^h_j}{\partial y^i}$, respectively.

\vspace{4pt}
In the following we set the definition of a (conic pseudo-) Finsler manifold which will be used throughout the paper.
\begin{definition}
A conic sub-bundle of $TM$ is a non-empty open subset $\mathcal{A}  \subset TM_0$ such that  $ \pi(\mathcal{A}) = M $ and 
$ \lambda v\in \mathcal{A}\quad  \forall v\in \mathcal{A} \text{ and } \,\forall \lambda\in\mathbb{R^{+}}$.
\end{definition}
\begin{definition}\label{pseudo-Finsler def}
  A \textbf{conic pseudo-Finsler  metric} on $ M $ is a function
  $\; F:\mathcal{A} \longrightarrow \mathbb{R}$
  which satisfies the following conditions:
  \begin{description}
    \item[(i)] $ F $ is smooth on $\mathcal{A}$, 
    \item[(ii)] $ F(x,y)$ is positively homogeneous of degree one in $y$: $F(x,\lambda y)=\lambda F(x,y)$ for all $(x,y)\in \mathcal{A}$  and $\lambda \in \mathbb{R^+},$
    \item[(iii)] The Finsler metric tensor $ g_{ij} (x,y)=\frac{1}{2}\dot{\partial}_{i}\dot{\partial}_{j}F^{2} (x,y)$
    is non-degenerate at each point of $ \mathcal{A}$.
  \end{description}
The pair $(M, F)$ is called a conic pseudo-Finsler manifold.

\vspace{6pt}
The Finsler metric $F$ is said to be:\\
$\bullet$ a \textbf{conic Finsler metric} if the  metric tensor $g_{ij}$ in $\textbf{\emph{(iii)}}$ is positive definite, or equivalently, $F$ is non-negative, and in this case $(M, F)$ is  a conic Finsler manifold.\\
$\bullet$ a \textbf{pseudo-Finsler metric} if $\mathcal{A}$ is replaced by $TM$, and in this case $(M, F)$ is  a pseudo-Finsler manifold.\\
$\bullet$ a \textbf{Finsler metric} if $\mathcal{A}$ is replaced by $TM$ and $g_{ij}$ in $\textbf{\emph{(iii)}}$ is positive definite,  and in this case $(M, F)$ is  a  Finsler manifold.
\end{definition}  
 
Each pseudo-Finsler metric $F$ induces a spray $S$ on $M$, for which the spray coefficients are given by $G^i = \frac{1}{4} g^{ij}(y^{k}\,\dot{\partial}_{j}\partial_{k} {F}^{2}-\partial_{j}{F}^{2})$ \cite{Gifone,Shenbook16}, this spray is called the  \emph{geodesic spray} of $F$.

%%%%%%%%%%%%%%%%%%%%%%%%%%% Sec. 2 %%%%%%%%%%%%%%%%%%%%%%%%%%%%%%%%

\section{Anisotropic conformal transformation}\label{section two}
In this section, we introduce the anisotropic conformal transformation of a conic pseudo-Finsler metric $F$ and investigate its elementary properties. 
  
Conformality is classically defined as follows.
 
\begin{definition}\emph{\cite{rund}}
Let $F$ and $\overline{F}$ be two Finsler metric functions defined on a smooth manifold $M$. The two Finsler metric tensors $g_{ij}$ and  $ \overline{g}_{ij}$ resulting from $F$ and $\overline{F}$, respectively, are called conformal if there exists a factor of proportionality $\Psi(x,y)$ between the metric tensors:  
\begin{equation}
\overline{g}_{ij}(x,y)=\Psi(x,y)\,g_{ij}(x,y),
\end{equation}\label{conformality} 
where $\Psi(x,y)$ is a positive smooth function of position $x$ and direction $y$.
\end{definition}

     Suppose $g_{ij}$ and $\overline{g}_{ij}$ are conformal, then (\ref{conformality}) holds. 
  As $g_{ij}y^{i}y^{j}=F^{2}$, we get
  \begin{equation}\label{Fis.fu.proportional}
    \overline{F}(x,y)=\sqrt{\Psi(x,y)}\,F(x,y).
  \end{equation}
  
  This means that if $g_{ij}(x,y)$ and $\overline{g}_{ij}(x,y)$ are proportional, then $F(x,y)$ and $\overline{F}(x,y)$ are proportional, or, in other words, \eqref{conformality} implies  \eqref{Fis.fu.proportional}. The converse is not true in general \cite{rund}: proportional Finsler metric functions may yield non proportional Finsler metric tensors. Nevertheless, we have
   
\begin{lemma}
  Let two Finsler metric functions $F$ and $\overline{F}$ on $M$ be proportional, i.e., $\overline{F}(x,y)=\Psi(x,y)\,F(x,y)$. A necessary and sufficient condition for the associated Finsler metric tensors $g_{ij}$ and $\overline{g}_{ij}$ to be proportional is that $\Psi(x,y)$ is a functions of position $x$ only.
\end{lemma}
\begin{proof}
 Let $F$ and $\overline{F}$ be proportional, that is,
  \begin{align*}
  \overline{F}(x,y)=\Psi(x,y)\,F(x,y).
\end{align*}   
  By squaring and differentiating with respect to $ y^{i} \text{ and }y^{j},$ we get
  \begin{equation*}
  \overline{g}_{ij}  =\Psi^2 g_{ij}+\{(\dot{\partial}_{i}\Psi)
  (\dot{\partial}_{j}\Psi)+\Psi(\dot{\partial}_{i}\dot{\partial}_{j}\Psi)\}F^2 + \Psi\{(\dot{\partial}_{j}\Psi)\dot{\partial}_{i}{F}^{2}
    +(\dot{\partial}_{i}\Psi)\dot{\partial}_{j}F^{2}\}.
  \end{equation*}
  Consequently, the metric tensors $g_{ij}$ and $\overline{g}_{ij}$ are proportional if and only if
  \begin{equation*}
  \{(\dot{\partial}_{i}\Psi)
  (\dot{\partial}_{j}\Psi)+\Psi(\dot{\partial}_{i}\dot{\partial}_{j}\Psi)\}F^2 + \Psi\{(\dot{\partial}_{j}\Psi)\dot{\partial}_{i}{F}^{2}
    +(\dot{\partial}_{i}\Psi)\dot{\partial}_{j}F^{2}\}=0.
  \end{equation*}
   Contracting both sides of the above equation by $ y^{j}$ and  using the homogeneity properties of $F$  and $\Psi$, we get
   $$\Psi(\dot{\partial}_{i}\Psi)F^{2}=0,$$
  from which $\dot{\partial}_{i}\Psi(x,y)=0$ and $\Psi\neq0$ is a function of $x$ only.
\end{proof}

\begin{remark}
\emph{\textbf{(a)}} The above discussion shows that the proportionality of $F$ and $\overline{F}$ \textbf{does not imply} automatically that the conformal factor $\Psi$ is independent of direction, contrary to what has been mentioned by Y. Shen and Z. Shen in \emph{\cite[(11.1.1)]{Shenbook16}}.\vspace{5pt}\\
\emph{\textbf{(b)}} On the other hand, it was proved by Knebelman \emph{\cite{Knebelman}} that the proportionality of $g_{ij}$ and $\overline{g}_{ij}$ \textbf{does imply} automatically that the conformal factor $\Psi$ is independent of direction, as has been mentioned by H. Rund in \emph{\cite[VI,2]{rund}}.  
\end{remark}

We begin our investigation from the proportionality of $F$ and $\overline{F}$, regardless of the proportionality of $g_{ij}$ and $\overline{g}_{ij}$ and regardless of the independence of the conformal factor of direction. In fact, we are particularly interested in the dependence of the conformal factor on both position and direction, motivated by the current applications in physics and other branches of science, and by the high potentiality of near future applications. 

Now, let us set our definition of conformality.
\begin{definition}
  Let $ (M, F) $ be a conic pseudo-Finsler  manifold.  The anisotropic conformal transformation of  $F $ is defined by
  \begin{equation}
  \label{Eq:Anisotropic_change}
  F\longmapsto\overline{F}(x,y)=e^{\phi(x,y)}F(x,y),
  \end{equation}
    where  $ \phi(x,y)$ is  an $h(0)$ smooth function on $\mathcal{A}$. In this case, we say that $F$ and $\overline{F}$ are anisotropically conformally related, or $\overline{F}$ is anisotropically conformal to $F$.
    
\noindent If the conformal factor $\phi(x,y)$ is independent of direction,  \eqref{Eq:Anisotropic_change} is called isotropic conformal, or, simply, conformal transformation.     
\end{definition} 

 One of the benefits of the anisotropic conformal transformation is that, unlike conformal transformation \cite{Hashiguchi76}, the anisotropic conformal transformation can transform a Riemannian metric to a Finslerian one. 

\begin{proposition}
  Let $(M,F)$ be a conic pseudo-Finsler space, then under the anisotropic conforaml transformation \eqref{Eq:Anisotropic_change}, we have
\begin{equation}\label{g.bar.metric.n.dim}
\overline{g}_{ij}=  e^{2\phi}\left[ g_{ij}+2 F^2 \,\dot{\partial_{i}}\phi \, \dot{\partial_{j}}\phi + 2F (\,{\ell}_{j}\, \dot{\partial_{i}}\phi+    {\ell}_{i}\,\dot{\partial_{j}} \phi)+F^2\, \dot{\partial}_{i}\dot{\partial}_{j}\phi \right].
\end{equation}
\end{proposition}
\begin{proof}
  Since $ \overline{F}^{2}=e^{2\phi(x,y)}F^{2}$, we get $\dot{\partial_{i}}\overline{F}^{2}=e^{2\phi}\,(2 \, F^{2}\, \dot{\partial_{i}}\phi+\dot{\partial_{i}}F^{2}).$
 Thereby,
\begin{flalign*}
 \overline{g}_{ij}    = \frac{1}{2} \dot{\partial_{i}} \dot{\partial_{j}}\overline{F}^{2} &=e^{2\phi}[2 F^{2}(\dot{\partial_{i}}\phi)(\dot{\partial_{j}}\phi) + 2F\,(\ell_{j}\, \dot{\partial_{i}}\phi+\ell_{i}\,\dot{\partial_{j}}\phi)+F^{2}\,\dot{\partial}_{i}\dot{\partial}_{j}\phi+\frac{1}{2} \dot{\partial_{i}}\dot{\partial_{j}}F^{2}]\\
     & =e^{2\phi}[ g_{ij}+2 F^{2}(\dot{\partial_{i}}\phi)(\dot{\partial_{j}}\phi) + 2F\,(\ell_{j}\, \dot{\partial_{i}}\phi+\ell_{i}\,\dot{\partial_{j}}\phi)+F^{2}\,\dot{\partial}_{i}\dot{\partial}_{j}\phi],\\
     \text{ where }  \ell_i=\dot{\partial}_iF.
\end{flalign*}
\vspace*{-1.3 cm}\[\qedhere\]
\end{proof}

\begin{remark}
The existence of the term $F^2\dot{\partial}_{i}\dot{\partial}_{j}\phi$ in \eqref{g.bar.metric.n.dim}  makes calculating the inverse metric $\overline{g}^{ij} $, in a tensorial form, very complicated in  general dimension. However, $\overline{g}^{ij}$ can be found out in the two-dimensional case thanks to the existence of Berwald frame in such a case.
 \end{remark} 
 
\textbf{Henceforward, we work in  a two-dimensional conic pseudo-Finsler space equipped with a modified Berwald frame $(\ell_i,m_i)$} \cite{Berwald41,Matsumoto 2003}.  The components $g_{ij}$ of the  metric tensor are given by
  \begin{align}\label{g_ij-2d}
  g_{ij}=\ell_{i}\ell_{j}+\varepsilon m_{i}m_{j}, 
  \end{align}
  where  $\varepsilon=\pm1$ is called the signature of $F$. Further,   the angular metric coefficients $h_{ij}$ can be expressed as 
  \begin{align*}\label{h_ij-2d} 
h_{ij}=\varepsilon m_{i}m_{j}.
\end{align*}
Also, we have 
\begin{equation*}
\label{g^ij-2d}
g^{ij} =\ell^{i}\ell^{j}+\varepsilon m^{i}m^{j},
\end{equation*}
where $\ell^i=\frac{y^i}{F}$. The two vector fields $\ell = (\ell^1,\ell^2)$ and $ m=(m^1 ,m^2)$ have been chosen in such a way that they satisfy $$ g(\ell, \ell)=1, \; g(\ell, m)=0,\;  g(m, m)=\varepsilon. $$  Moreover, the determinant $\mathfrak{g} $ of the matrix $(g_{ij})$ is given by
\begin{equation}
\label{Eq:det(g)}
\mathfrak{g} :=\det(g_{ij})=\varepsilon(\ell_{1}m_{2}-\ell_{2}m_{1})^2.
\end{equation}
 To calculate $m^i $ and    $m_i$, we give the following lemma.
 
 \begin{lemma}
 The covector $m_i$ is given by
 $$m_1=-\sqrt{\varepsilon \mathfrak{g}} \ \ell^2, \quad m_2=\sqrt{\varepsilon \mathfrak{g}} \ \ell^1.$$
 Moreover, the   vector $m^i$ is given by
 $$m^1=-\frac{1}{\sqrt{\varepsilon \mathfrak{g}}}\,\ell_2, \quad m^2=\frac{1}{\sqrt{\varepsilon \mathfrak{g}}}\,\ell_1. $$
 \end{lemma}
 \begin{proof}
 By \eqref{Eq:det(g)}, we have
 $$\ell_{1}m_{2}-\ell_{2}m_{1}=\sqrt{\varepsilon    \mathfrak{g}}.$$
 Also,  the property $g(\ell ,m)=0$  gives
 $$\ell^1m_1+\ell^2m_2=0.$$
Then, the result follows by solving the above  two equations for $m_1$ and $m_2$.  The proof of the formulae of $m^i$ can be calculated in a similar manner. 
 \end{proof}
 
Thus, the components $C_{ijk}$ of the Cartan tensor are given by \cite{Berwald41} 
\begin{equation*}\label{cartan tensor Finsler surface} 
 F C_{ijk}=\mathcal{I}\;m_{i}m_{j}m_{k},
\end{equation*}
where $ \mathcal{I}$ is the  main scalar of the surface $(M,F)$.

\vspace{5pt}
From now on, when we say that $(M,F)$ is a conic pseudo-Finsler surface, this means that $(M,F)$ is a conic pseudo-Finsler surface equipped with the modified Berwald frame.
\begin{lemma}\em\cite{Matsumoto 2003}\label{properties.2.dim.Fins}
  \it Let $(M,F)$ be a conic pseudo-Finsler surface. Then we have the following:
  \begin{description}
    \item[(i)] $\ell^{i}m_{i}=\ell_{i}m^{i}=0$,
     \item[(ii)] $m^{i}m_{i}=\varepsilon$,$\quad \ell^i\ell_i=1$,
     \item[(iii)] $\delta_{i}^{j}=\ell_{i}\ell^{j}+\varepsilon m_{i}m^{j}$,
    \item[(iv)] $F\dot{\partial _{j}}\ell_{i}=\varepsilon m_{i}m_{j}=h_{ij},\qquad F\dot{\partial _{j}}\ell^{i}=\varepsilon m^{i}m_{j} $,
    \item[(v)] $F\dot{\partial _{j}}m_{i}=-(\ell_{i}-\varepsilon \mathcal{I} m_{i})m_{j}$,\qquad$F\dot{\partial _{j}}m^{i}=-(\ell^{i}+\varepsilon \mathcal{I} m^{i})m_{j}$,
     \item[(vi)] For a smooth function $f$ on $\mathcal{A}$, we have $F\dot{\partial _{i}}f= f_{;1}\,\ell_{i} + f_{;2}\,m_{i}$,  where $f_{;1}=F(\dot{\partial _{i}}f) \,\ell^{i}$ and $f_{;2}=\varepsilon F\,(\dot{\partial _{i}}f) m^{i}.$ In particular, if $ f$ is $ h(r) $, then  $f_{;1}=rf$,
    \item[(vii)] $F\dot{\partial _{i}}\partial_{k}f= (\partial_{k}f)_{;1}\,\ell_{i} + (\partial_{k}f)_{;2}\,m_{i}$, where $\partial_k=\frac{\partial}{\partial x^k}$,
    \item[(viii)] $\delta _{i}f= f_{,1}\,\ell_{i} + f_{,2}\,m_{i},$ where $  f_{,1}= (\delta_if)\ell^i$ and  $f_{;2}=\varepsilon (\delta_{i}f)m^{i}$.
\end{description}
\end{lemma}

Using the modified  Berwald frame, we can write the formulae of the objects $\dot{\partial}_{i}\phi$ and  $\dot{\partial}_{i}\dot{\partial}_{j}\phi$ in terms of the  frame elements.  Consequently,  plugging these objects into  \eqref{g.bar.metric.n.dim},  we are able to  find the formula of components $\overline{g}^{ij}$ of the inverse metric of $\overline{F}$.  More precisely, we have:

\begin{lemma}\label{tensor of second direviative of phi}
   Let $(M,F)$ be a  conic pseudo-Finsler surface. The term  $\dot{\partial_{i}}\dot{\partial_{j}}\phi$ given in \eqref{g.bar.metric.n.dim} can be expressed in the form 
   \begin{align}\label{F^2.phi_ij-2d}
   F^{2}\dot{\partial}_{i}\dot{\partial}_{j}\phi =-\phi_{;2}\,(\ell_{i}m_{j}+\ell_{j}m_{i})+( \phi_{;2;2}+\varepsilon\, \mathcal{I}\,\phi_{;2})\,m_{j}m_{i}.
   \end{align}
\end{lemma}
\begin{proof}
From Lemma \ref{properties.2.dim.Fins} $\textbf{(vi)} ,$  
$F\dot{\partial}_{i}\phi=\phi_{;1}\;\ell_{i}+\phi_{;2}\,m_{i}.$ Since $\phi$ is $h(0)$, then $\phi_{;1}=0$. That is,
\begin{gather}\label{phi_i.2.dim}
   F\dot{\partial}_{i}\phi=\phi_{;2}\;m_{i}.
\end{gather}
Multiplying both sides of \eqref{phi_i.2.dim} by $F$, then differentiating with respect to $y^{j}$ and use Lemma \ref{properties.2.dim.Fins} \textbf{(v)}  and \textbf{(vi)}, we get
\begin{equation*}
  F^{2}\dot{\partial}_{i}\dot{\partial}_{j}\phi=-\phi_{;2}\,(\ell_{i}m_{j}+\ell_{i}m_{j})+( \phi_{;2;2}+\varepsilon\, \mathcal{I}\,\phi_{;2})\,m_{j}m_{i}.
\end{equation*}
\vspace*{-1.3cm}\[\qedhere\]
\end{proof}

\begin{proposition}\label{lemma of g_ij}
  Let $(M,F)$ be a conic pseudo-Finsler surface, then under the anisotropic conformal transformation \eqref{Eq:Anisotropic_change},    the   metric tensors of $\overline{F}$ and $F$ are related by
  \begin{align}\label{2.form.g_ij.bar.2.dim}
     \overline{g}_{ij} & =e^{2\phi}[ g_{ij}+\phi_{;2}(\ell_{i}m_{j}+\ell_{j}m_{i})+\sigma m_{i}m_{j}]\nonumber \\
       & = e^{2\phi}[ \ell_{i}\ell_{j}+\phi_{;2}(\ell_{i}m_{j}+\ell_{j}m_{i})+(\sigma+\varepsilon) m_{i}m_{j}],
\end{align}
where $\sigma=\phi_{;2;2}+\varepsilon \mathcal{I}\phi_{;2}+2\,(\phi_{;2})^{2}$.
\end{proposition}
\begin{proof}
  It follows from \eqref{g.bar.metric.n.dim}, \eqref{g_ij-2d}, \eqref{F^2.phi_ij-2d} and \eqref{phi_i.2.dim}.
\end{proof}

It should be noted that the anisotropic conformal transformation of a conic pseudo-Finsler surface $(M,F)$ does not yield in general a conic pseudo-Finsler surface, as shown in Example \ref{example of non pseudo finsler Fbar} below.
The following result gives a necessary and sufficient condition for $(M,\overline{F})$ to be a conic pseudo-Finsler surface. 
 
\begin{theorem}\label{ness.suff.codtion.F.bar.Finsler}
Let $(M,F)$ be a conic pseudo-Finsler surface and  \eqref{Eq:Anisotropic_change} be an anisotropic conformal transformation of $F$.    Then, $(M,\overline{F})$ is a conic pseudo-Finsler surface  if and only if $$\sigma-(\phi_{;2})^{2}+\varepsilon=F^{2}[\dot{\partial}_{i}\dot{\partial}_{j}\phi
 +(\dot{\partial}_{i}\phi)(\dot{\partial}_{j}\phi)] m^{i}m^{j}+\varepsilon\neq0.$$
\end{theorem}
\begin{proof}
 Since both  $F(x,y)$ and  $e^{\phi(x,y)}$ are smooth on $ \mathcal{A}$,  then  $\overline{F}(x,y)$ is smooth on $ \mathcal{A}$. Moreover,   $\overline{F}(x,y)$ is $h(1)$. Now, from \eqref{2.form.g_ij.bar.2.dim}, we get

  \small{$$ \det(\overline{g}_{ij})=e^{4\phi}\begin{vmatrix}
                          \ell_{1}\ell_{1}+2\phi_{;2}\, \ell_{1}m_{1}+(\sigma+\varepsilon)\, m_{1}m_{1} & \ell_{1}\ell_{2}+\phi_{;2} (\ell_{1}m_{2}+ \ell_{2}m_{1}) +(\sigma+\varepsilon) m_{1}m_{2} \\
                          \ell_{1}\ell_{2}+\phi_{;2}(\ell_{1}m_{2}+ \ell_{2}m_{1}) +(\sigma+\varepsilon) m_{1}m_{2} & \ell_{2}\ell_{2}+2\phi_{;2} \,\ell_{2}m_{2}+(\sigma+\varepsilon)\, m_{2}m_{2} 
                        \end{vmatrix}.$$} 
That is, by using \eqref{Eq:det(g)}, we have
\begin{eqnarray}
    \det(\overline{g}_{ij}) &  =&e^{4\phi}[(\sigma+\varepsilon) \ell_{1}\ell_{1} m_{2}m_{2}+(\sigma+\varepsilon) \ell_{2}\ell_{2} m_{1}m_{1}-2(\sigma+\varepsilon) \ell_{1}m_{1}\ell_{2}m_{2} \nonumber \\
    &  &\quad \,\, -(\phi_{;2})^{2} \ell_{1}\ell_{1} m_{2}m_{2}-(\phi_{;2})^{2}  \ell_{2}\ell_{2} m_{1}m_{1}+2(\phi_{;2})^{2}\ell_{1}m_{1} \ell_{2}m_{2}]\nonumber \\
     & =& e^{4\phi}[\sigma+\varepsilon-(\phi_{;2})^{2}](\ell_{1}m_{2}-\ell_{2}m_{1})^{2}\nonumber\\
   &=&\varepsilon e^{4\phi}[\sigma-(\phi_{;2})^{2}+\varepsilon]\det(g_{ij}).\nonumber
 \end{eqnarray}
 Therefore, the matrix  $(\overline{g}_{ij})$ is non-degenerate if and only if
$$ \sigma-(\phi_{;2})^{2}+\varepsilon=F^{2}[\dot{\partial}_{i}\dot{\partial}_{j}\phi\
 +(\dot{\partial}_{i}\phi)(\dot{\partial}_{j}\phi)] m^{i}m^{j}+\varepsilon\neq0.$$
 \vspace*{-1.3cm}\[\qedhere\]
 \end{proof}

\begin{example}\label{example of non pseudo finsler Fbar}
This example gives a conic pseudo-Finsler surface $F$ whose anisotropic conformal transformation $\overline{F}=e^{\phi} F$ is not a conic pseudo-Finsler surface for a certain smooth $h(0)$ function~ $\phi$. Let $F(x,y)=\dfrac{\sqrt{|y|^2-(|x|^2|y|^2-\langle x, y\rangle^2)}}{1-|x|^2}$ be the Klein metric on unit ball $B^n\subset\mathbb{R}^{n}$. Define
$$\overline{ F}(x,y)= e^{\phi (x,y)} F(x,y)=\dfrac{\langle x, y\rangle}{1-|x|^2}, $$
where $\phi (x,y)=\ln \left(\dfrac{\langle x, y\rangle}{\sqrt{|y|^2-\left(|x|^2|y|^2-\langle x, y\rangle^2\right)}} \right).$\\
It is clear that $\overline{F}$ is linear function in $y$, i.e.,   $\det(\overline{g}_{ij})=0$. That is,  $\sigma+\varepsilon-(\phi_{;2})^{2}=0.$ Hence, $\overline{F}$ is not a conic pseudo-Finsler metric, by Theorem \emph{\ref{ness.suff.codtion.F.bar.Finsler}}.
\end{example}

\begin{proposition}\label{prop of g^ij}
Let\; $ (M,F)$ be a  conic pseudo-Finsler surface and \eqref{Eq:Anisotropic_change} be an anisotropic conformal transformation of $F$. Then, we have
\begin{align}\label{g^ij.bar.2.dim}
  \begin{split}
     \overline{g}^{ij}  & =e^{-2\phi}[\varepsilon \rho g^{ij}+\sigma\rho\ell^{i}\ell^{j}
-\varepsilon\phi_{;2}\; \rho(\ell^{i}m^{j}+\ell^{j}m^{i})]\\
& =e^{-2\phi}[g^{ij}+(\phi_{;2})^{2}\rho\ell^{i}\ell^{j}
-\varepsilon\phi_{;2}\;\rho(\ell^{i}m^{j}+\ell^{j}m^{i})+(\rho-\varepsilon) m^{i}m^{j}],
  \end{split}
\end{align}
where
$\rho:=\dfrac{1}{\sigma+\varepsilon-(\phi_{;2})^{2}}.$
\end{proposition}
\begin{proof}
   In a conic pseudo-Finsler surface $(M,F)$, a tensor of type (2,0) can be written as $ T^{ij}=A_{1}\ell^{i}\ell^{j}+A_{2}\ell^{i}m^{j}+A_{3}m^{i}\ell^{j}+A_{4}m^{i}m^{j},$ where $A_{1}, \;A_{2},\; A_{3} \text{ and }  A_{4}$ are smooth function on $\mathcal{A}$.
 Consequently,  the inverse of the metric tensor $\overline{g}_{ij}$, given by \eqref{2.form.g_ij.bar.2.dim}, can be written as
 \begin{equation}\label{to-dete-A-i}
   \overline{g}^{ij}=e^{-2\phi}[A_{1}\ell^{i}\ell^{j}+A_{2}\ell^{i}m^{j}+A_{3}m^{i}\ell^{j}+A_{4}m^{i}m^{j}].
 \end{equation}
To determine the functions $A_{1},\; A_{2},\; A_{3} \text{ and }  A_{4}$, we calculate
 $$ \overline{g}^{ir}\overline{g}_{rj} =\delta_{j}^{i}=[A_{1}\ell^{i}\ell^{r}+A_{2}\ell^{i}m^{r}+A_{3}m^{i}\ell^{r}+A_{4}m^{i}m^{r}][ \ell_{r}\ell_{j}+\phi_{;2}(\ell_{r}m_{j}+\ell_{j}m_{r})+(\sigma+\varepsilon) m_{r}m_{j}]. $$
Using Lemma \ref{properties.2.dim.Fins} (\textbf{(i)-(iii)}), we get
\begin{align}\label{constant-g^{ij}}
  0= &  (A_{1}+\varepsilon\phi_{;2} \;A_{2}-1)\ell_{j}\ell^{i}+(A_{3}+\varepsilon\phi_{;2}\; A_{4})\ell_{j}m^{i}+(\phi_{;2}\;A_{1}+\varepsilon A_{2}(\sigma+\varepsilon))m_{j}\ell^{i}\nonumber \\
  &  +(\phi_{;2}\; A_{3}+\varepsilon A_{4}(\sigma+\varepsilon)-\varepsilon) m_{j}m^{i}.
\end{align}
 Contracting \eqref{constant-g^{ij}} by $\ell^{i}\ell_{j}, m^{i}m_{j}, \ell^{i}m_{j} \text{ and } m^{i}\ell_{j}, $ respectively, we get the following equations:
\begin{equation}\label{system-cons}
A_{1}+\varepsilon\phi A_{2}   =1, \quad
    A_{3}+\varepsilon\phi_{;2}\; A_{4}  =0, \quad
    \phi_{;2}\;A_{1}+\varepsilon(\sigma+\varepsilon) A_{2}   =0, \quad
    \phi_{;2}\; A_{3}+\varepsilon(\sigma+\varepsilon) A_{4}  =\varepsilon.
\end{equation}
 Solving the system  \eqref{system-cons} yields
 \begin{equation*}
  A_{1}=\frac{\sigma +\varepsilon}{\sigma +\varepsilon-(\phi_{;2})^{2}},\hspace{1.5cm} A_{2}=A_{3}=\frac{ -\varepsilon\phi_{;2}}{\sigma +\varepsilon-(\phi_{;2})^{2}},\hspace{1.5cm}A_{4}=\frac{1}{\sigma +\varepsilon-(\phi_{;2})^{2}}.
 \end{equation*}
Substituting  $A_{1},\; A_{2},\; A_{3} \text{ and }  A_{4}$ into \eqref{to-dete-A-i}, noting that $\rho: =\frac{1}{\sigma +\varepsilon-(\phi_{;2})^{2}}$, we obtain
\begin{equation*}
\overline{g}^{ij}=e^{-2\phi}[\rho(\sigma+\varepsilon)\ell^{i}\ell^{j}
-\varepsilon\phi_{;2}\;\rho(\ell^{i}m^{j}+\ell^{j}m^{i})+\rho m^{i}m^{j}].
\end{equation*}
By \eqref{g_ij-2d},  we get
$$\overline{g}^{ij}=e^{-2\phi}[\varepsilon \rho g^{ij}+\sigma\rho\ell^{i}\ell^{j}
-\varepsilon\phi_{;2} \;\rho(\ell^{i}m^{j}+\ell^{j}m^{i})].$$
\vspace*{-1.1cm}\[\qedhere\]
\end{proof}

Let us end this section by listing some properties of the functions $\phi$,  $\rho $, $\sigma$ and $\mathcal{I}$ for subsequent use. 

\begin{remark}\label{remark phi and rho }
As\; $\phi_{;2},\; \sigma,\;\rho$ and $\mathcal{I}$ are $h(0)$-functions, the following relations hold:
  \begin{description}
    \item[(a)]$\sigma=\phi_{;2;2}+\varepsilon \mathcal{I}\phi_{;2}+2(\phi_{;2})^{2}$,\quad $\rho=\dfrac{1}{\sigma+\varepsilon-(\phi_{;2})^{2}}=\dfrac{1}{\varepsilon+\phi_{;2;2}+\varepsilon \mathcal{I}\phi_{;2}+(\phi_{;2})^{2}}$.
    \item[(b)] $F\dot{\partial}_{i}\phi_{;2}=\phi_{;2;2}\;m_{i}$,\quad $ F\dot{\partial}_{i}\sigma=\sigma_{;2}m_{i},\quad  F\dot{\partial}_{i}\rho=\rho_{;2}\,m_{i}$, \quad   
    $ F\dot{\partial}_{i}\mathcal{I}=\mathcal{I}_{;2}\;m_{i}$.  
  \end{description}
\end{remark}

%%%%%%%%%%%%%%%%%%%%%%%%%%% Sec. 3 %%%%%%%%%%%%%%%%%%%%%%%%%%%%%%%

\section{Transformation of fundamental geometric objects }

 From now on, we consider the anisotropic conformal transformation of  conic pseudo-Finsler metric $F$ given by
 \begin{equation}\label{the anisotropic conformal transformation2}
 \overline{F}(x,y)=e^{\phi(x,y)}F(x,y), \quad\text{with}\,\,\, \sigma+\varepsilon-(\phi_{;2})^{2}\neq0.
 \end{equation}
 
We calculate the anisotropic conformal change \eqref{the anisotropic conformal transformation2} of some  important geometric objects. Namely, we find the transformation of modified Berwald frame, angular metric, Cartan tensor, main scalar, geodesic spray, Barthel connection and  Berwald connection.

\begin{proposition}\label{tras-of.l_i.l^j.c_ijk}
  Under the anisotropic conformal transformation \eqref{the anisotropic conformal transformation2}, we obtain
  \begin{description}
    \item[(i)] $\overline{\ell}_{i}=e^{\phi}[\ell_{i}+\phi_{;2}\; m_{i}],$\;\;\;\;\;$\overline{\ell}^i=e^{-\phi}\ell^{i},$
    \item[(ii)]$ \overline{h} _{ij}=e^{2\phi}[h_{ij}+(\sigma-(\phi_{;2})^{2}) m_{i}m_{j}],$
    \item[(iii)] $ \overline{m}_{i}=e^{\phi}\sqrt{\frac{\varepsilon}{\rho}}\;m_{i},$\;\;\;\;\;$\overline{m}^{j}=e^{-\phi}\sqrt{\varepsilon\rho}[m^{j}-\varepsilon \phi_{;2}\;\ell^{j}],$
    \item[(iv)] $\overline{m}_{i}\overline{m}^{i}=\varepsilon,$\;\;\;\;\;$\overline{\ell}_{i}\overline{\ell}^{i}=1,$\;\;\;\;\;$ \overline{\ell}^{i} \overline{m}_{i}=\overline{\ell}_{i}\overline{m}^{i}=0.$
  \end{description}
\end{proposition}
\begin{proof}
\textbf{(i)} Since $\overline{F}=e^{\phi}F, $ by \eqref{phi_i.2.dim}, we get
\begin{equation}\label{differentiate F w.r.t y^i}
  \dot{\partial_{i}}\overline{F}=\overline{\ell}_{i}=e^{\phi}F\dot{\partial_{i}}\phi +e^{\phi} \, \dot{\partial_{i}}F=e^{\phi}[\ell_{i}+F\dot{\partial_{i}}\phi]=e^{\phi}[\ell_{i}+\phi_{;2}m_{i}].
  \end{equation}
Then from \eqref{g^ij.bar.2.dim}, we obtain
\begin{flalign*}
     \overline{\ell}^{j}  :=\overline{g}^{ij}\,\overline{\ell}_{i}&=e^{-\phi}[\rho(\sigma+\varepsilon)\ell^{i}\ell^{j}-\varepsilon\phi_{;2}\;\rho(\ell^{i}m^{j}+\ell^{j}m^{i})+\rho m^{i}m^{j}](\ell_{i}+\phi_{;2}\; m_{i}) \\
     & =e^{-\phi}\rho \,\ell^{j}(\sigma+\varepsilon-(\phi_{;2})^{2})=e^{-\phi}\,\ell^{j}.
\end{flalign*}

\textbf{(ii}) Differentiating \eqref{differentiate F w.r.t y^i} with respect to $ y^{j}$, we get
\begin{equation*}
\dot{\partial_{i}}\dot{\partial_{j}}\overline{F}= e^{\phi}[F(\dot{\partial_{i}}\phi)(\dot{\partial_{j}}\phi)+F\dot{\partial_{i}}\dot{\partial_{j}}\phi+(\dot{\partial_{i}}F)\dot{\partial_{j}}\phi+(\dot{\partial_{j}}F)\dot{\partial_{i}}\phi+\dot{\partial_{i}}\dot{\partial_{j}}F].
\end{equation*}
From \eqref{F^2.phi_ij-2d} and \eqref{phi_i.2.dim}, we obtain 
$$\overline{h}_{ij}  =\overline{F}\,\dot{\partial_{i}}\dot{\partial_{j}}\overline{F} =e^{2\phi}[h_{ij}+(\sigma-(\phi_{;2})^{2})m_{i}m_{j}].$$

\textbf{(iii)} From  \textbf{(ii)}, we have
 $\varepsilon \overline{m}_{i}\overline{m}_{j} =e^{2\phi}[ \sigma+\varepsilon-(\phi_{;2})^{2}] m_{i}m_{j} = \frac{ e^{2\phi}}{\rho}\, m_{i}m_{j}.$\\
 Hence,  $ \overline{m}_{i}=\sqrt{\frac{\varepsilon}{\rho}}\,e^{\phi}\,m_{i}$.
 But $m^{i}=g^{ij}m_{j}$, then from \eqref{g^ij.bar.2.dim}, we have
    \begin{equation*}
      \overline{m}^{j}   = \sqrt{\varepsilon\rho}\,e^{-\phi}\,[m^{j}-\varepsilon\phi_{;2} \;\ell^{j}].
    \end{equation*}
    
\textbf{(iv)} It follows from \textbf{(i)} and \textbf{(iii)}.
  \vspace*{-.45cm}\[\qedhere\]
 
\end{proof}

\vspace*{.3cm}

In view of Proposition \ref{tras-of.l_i.l^j.c_ijk} \textbf{(iv)}, we have the following corollary.
\begin{corollary}
Let $(M,F)$ be a conic pseudo-Finsler surface. The anisotropic conformal transformation  of a modified Berwald frame is a modified Berwald frame.
\end{corollary}

Now, we find the anisotropic conformal transformation of some  non-Riemannian quantities.
 \begin{proposition}\label{transform of cartan}
Under the anisotropic conformal transformation \eqref{the anisotropic conformal transformation2}, we  get
  \begin{description}
    \item[(i)] $F\;\overline{C}_{ijk} = e^{2\phi}F\,C_{ijk}+e^{2\phi}[\varepsilon \mathcal{I}\sigma+\frac{1}{2}\sigma_{;2}+\phi_{;2}(\sigma+2\varepsilon)]m_{i}m_{j}m_{k}$\\
  $\hspace*{1cm} =e^{2\phi}[\mathcal{I}(1+\varepsilon \sigma)+\frac{1}{2}\sigma_{;2}+\phi_{;2}(\sigma+2\varepsilon)]m_{i}m_{j}m_{k},$
    \item[(ii)] $\overline{\mathcal{I}}=(\varepsilon\rho)^{\frac{3}{2}}[\mathcal{I}(1+\varepsilon \sigma)+\frac{1}{2}\sigma_{;2}+\phi_{;2}(\sigma+2\varepsilon)],$
    \item[(iii)]$F\overline{C}_{jk}^{i}=\varepsilon\rho[F C_{jk}^{i}+\lbrace\varepsilon \mathcal{I}\sigma+\frac{1}{2}\sigma_{;2}+\phi_{;2}(\sigma+2\varepsilon)\rbrace m^{i}-\varepsilon\phi_{;2}\lbrace\mathcal{I}(1+\varepsilon\sigma)+\frac{1}{2}\sigma_{;2}+\phi_{;2}(\sigma+2\varepsilon)\rbrace\ell^{i}]m_{j}m_{k}.$
      \end{description}
    \end{proposition}
\begin{proof}
 
\textbf{(i)} Differentiating \eqref{2.form.g_ij.bar.2.dim} with respect to $y^{k}$, we have 
  \begin{flalign*}
    2\overline{C}_{ijk} =&\dot{\partial_{k}}(\overline{g}_{ij})\\
    =&2e^{2\phi}\dot{\partial}_{k}\phi[\ell_{i}\ell_{j}+\phi_{;2}(\ell_{i}m_{j}+\ell_{j}m_{i})+(\sigma+\varepsilon)m_{i}m_{j}]\\
    & +e^{2\phi}\dot{\partial}_{k}[g_{ij}+\phi_{;2} (\ell_{i}m_{j}+\ell_{j}m_{i})+\sigma m_{i}m_{j}]. 
    \end{flalign*} 
    Multiplying both sides of the above equation by $F$ and using Lemma \ref{properties.2.dim.Fins} \textbf{(iv)} and \textbf{(v)}, we have
    \begin{flalign*}   
    2F\overline{C}_{ijk} & =e^{2\phi} [2\phi_{;2}\ell_{i}\ell_{j}m_{k}+2(\phi_{;2})^{2} (\ell_{i}m_{j}m_{k}+\ell_{j}m_{i}m_{k})+2\phi_{;2}(\sigma+\varepsilon)m_{i}m_{j}m_{k}] \\
    & \;\;+2Fe^{2\phi}C_{ijk}+e^{2\phi}[\phi_{;2;2}(\ell_{i}m_{j}m_{k}+\ell_{j}m_{i}m_{k})+\phi_{;2}(\varepsilon m_{i}m_{j}m_{k}-\ell_{i}\ell_{j}m_{k}\\
    &  \;\;+\varepsilon \mathcal{I}\ell_{i}m_{j}m_{k}+\varepsilon m_{i}m_{j}m_{k} -\ell_{i}\ell_{j}m_{k}+\varepsilon \mathcal{I}\ell_{j}m_{i}m_{k})+\sigma_{;2}m_{i}m_{j}m_{k} \\
    &\;\; +\sigma(-\ell_{i}m_{j}m_{k}+\varepsilon \mathcal{I}m_{i}m_{j}m_{k}-\ell_{j}m_{i}m_{k}+\varepsilon \mathcal{I}m_{i}m_{j}m_{k})].
  \end{flalign*}
  By Remark \ref{remark phi and rho }, we get
\begin{equation*}
  F\overline{C}_{ijk} =F e^{2\phi}C_{ijk}+e^{2\phi}(\varepsilon \mathcal{I}\sigma+\frac{1}{2}\sigma_{;2}+\phi_{;2}(\sigma+2\varepsilon))m_{i}m_{j}m_{k}.
\end{equation*}

\textbf{(ii})  Since we have $ \overline{F}\;\overline{C}_{ijk}=\overline{\mathcal{I}}\;\overline{m}_{i}\;\overline{m}_{j}\;\overline{m}_{k},$ then by \textbf{(i)}
and Proposition \ref{tras-of.l_i.l^j.c_ijk} \textbf{(iii)}, we obtain
\begin{equation*}
  \overline{\mathcal{I}}e^{3\phi}\left(  \frac{\varepsilon}{\rho}\right) ^{\frac{3}{2}}m_{i}m_{j}m_{k}=e^{3\phi}[\mathcal{I}(1+\varepsilon \sigma)+\frac{1}{2}\sigma_{;2}+\phi_{;2}(\sigma+2\varepsilon)]m_{i}m_{j}m_{k}.
\end{equation*}
That is, 
\begin{equation*}\label{mean scalar F bar}
  \overline{\mathcal{I}}=(\varepsilon\rho)^{\frac{3}{2}}[\mathcal{I}(1+\varepsilon \sigma)+\frac{1}{2}\sigma_{;2}+\phi_{;2}(\sigma+2\varepsilon)].
\end{equation*}

\textbf{(iii)} By using \eqref{g^ij.bar.2.dim} and  \textbf{(i)}, we can write
    \begin{align*}
    F\overline{C}_{jk}^{i}  =&\overline{g}^{ih}\overline{C}_{hjk} =[\varepsilon\rho g^{ih}+\sigma\rho \ell^{i}\ell^{h}
-\varepsilon\phi_{;2} \rho(\ell^{i}m^{h}+\ell^{h}m^{i})] \\
    &\qquad \qquad \,\, [F C_{ijk}+(\varepsilon \mathcal{I}\sigma+\frac{1}{2}\sigma_{;2} +\phi_{;2}(\sigma+2\varepsilon))m_{i}m_{j}m_{k}]\\
     =& \varepsilon\rho[F C_{jk}^{i}+(\varepsilon \mathcal{I}\sigma+\frac{1}{2}\sigma_{;2}+\phi_{;2}(\sigma+2\varepsilon))m^{i}m_{j}m_{k}\\
     &\;\;-\varepsilon\phi_{;2} (\mathcal{I}(1+\varepsilon \sigma)+\frac{1}{2}\sigma_{;2}+\phi_{;2}(\sigma+2\varepsilon))\ell^{i}m_{j}m_{k}].
    \end{align*}
  	\vspace*{-1.5cm}\[\qedhere\]
 
\end{proof}

\begin{remark}

\emph{\textbf{(a)}} Substituting  $ C^{i}_{jk}= g^{ih}C_{hjk}=\frac{\mathcal{I}}{F}m^{i}m_{j}m_{k}$ in Proposition \emph{\ref{transform of cartan}} \emph{ \textbf{(iii)}}, we get
  \begin{equation*}\label{1,2 tensor of cartan}
    F\overline{C}^{i}_{jk}=\rho(\mathcal{I}(1+\varepsilon \sigma)+\frac{1}{2}\sigma_{;2}+\phi_{;2}(\sigma+2\varepsilon))(\varepsilon m^{i}m_{j}m_{k}-\phi_{;2}\ell^{i}m_{j}m_{k}).
  \end{equation*} 
  
\emph{\textbf{(b)}} From Lemma \emph{\ref{properties.2.dim.Fins}} \emph{\textbf{(v)}}  and Proposition \emph{\ref{tras-of.l_i.l^j.c_ijk}}, we can get another equivalent formula of the main scalar $\overline{\mathcal{I}}$ of $\overline{F}$   as follows
 \begin{align}\label{mean scalar F bar formula 2}
  \overline{\mathcal{I}}=\sqrt{\varepsilon\rho}\;[\mathcal{I}+2\varepsilon\phi_{;2}-\frac{\varepsilon\rho_{;2}}{2\rho}].
\end{align}
The equivalence of Proposition \emph{\ref{transform of cartan}} \emph{\textbf{(ii)}} and \eqref{mean scalar F bar formula 2} can be proven easily.

\emph{\textbf{(c)}} From \eqref{mean scalar F bar formula 2}, $\overline{\mathcal{I}}=\sqrt{\varepsilon\rho}\mathcal{I}$ if and only if  $4\rho\phi_{;2}=\rho_{;2}$. Hence, a necessary and sufficient condition for the property of being Riemannian space to be preserved under anisotropic conformal change is that   $$4\rho\phi_{;2}=\rho_{;2}.$$
\end{remark}

\medskip
In the following we find the change of the geodesic spray under the anisotropic conformal transformation \eqref{the anisotropic conformal transformation2}  and  consequently we find  the transformation of the nonlinear connection and Berwald connection.
 
\begin{proposition}\label{transformation of spray}
Under the anisotropic conformal transformation \eqref{the anisotropic conformal transformation2},  the geodesic spray coefficients $ \overline{G}^{i}$  and $G^{i}$ of $ \overline{F}$ and $F,$ respectively, are related by
\begin{equation}\label{coeff of sprayof F bar}
  \overline{G}^{i}= G^{i}+Q \,m^{i}+P\,\ell^{i},
\end{equation}
where
\begin{align}
P=&\frac{1}{4}[-\phi_{;2}\;A+2 F^{2}(\partial_{k}\phi)\ell^{k}]), \;\quad\;Q=\frac{1}{4}[\varepsilon A+2\varepsilon F(\partial_{k}F)m^{k}-2F(\partial_{k}F)_{;2}\ell^{k}],\label{formula for P, Q}\\
A=&\{2F\phi_{;2}\,\rho \;(F\partial_{k}\phi+\partial_{k}F)+2F\rho(F\partial_{k}\phi+\partial_{k}F)_{;2}\}\ell^{k}-2\varepsilon F\rho (F\partial_{k}\phi+\partial_{k}F)m^{k}.\nonumber
\end{align}
\end{proposition}
\begin{proof}
Since $ \overline{F}^{2}=e^{2\phi}F^{2},$ we get
$\partial_{k}\overline{F}^{2} =2Fe^{2\phi}[F\;\partial_{k}\phi+\partial_{k}F]. \;\text{Thus,}$
\begin{equation*}
\dot{\partial}_{j}\partial_{k}\overline{F}^{2}=e^{2\phi}[4\phi_{;2}( F\partial_{k}\phi+\partial_{k}F)m_{j}+4F(\partial_{k}\phi)\ell_{j}+2F(\partial_{k}\phi)_{;2}\;m_{j}
  +\dot{\partial}_{j}\partial_{k}F^{2}].
\end{equation*}
Consequently,
\begin{flalign}\label{G_K Ceoff-1form}
  y^{k}\dot{\partial}_{j}\partial_{k}\overline{F}^{2}-\partial_{j}\overline{F}^{2} & =e^{2\phi}[y^{k}\dot{\partial}_{j}\partial_{k}F^{2}-\partial_{j}F^{2}+(4\phi_{;2}( \;F^{2}\partial_{k}\phi+F\partial_{k}F)+2F^{2}(\partial_{k}\phi)_{;2})m_{j}\ell^{k}
  \nonumber \\
  & \quad +4F^{2}(\partial_{k}\phi)\ell_{j}\ell^{k}-2F^{2}\partial_{j}\phi].
\end{flalign}
Since $\partial_{k}F^{2}$ is $h(2),$  Eqn. \eqref{G_K Ceoff-1form}  can be written as
\begin{flalign}\label{G_k coeff 2 form}
  y^{k}\dot{\partial}_{j}\partial_{k}\overline{F}^{2}-\partial_{j}\overline{F}^{2}  & =e^{2\phi}[\{4\phi_{;2}(F^{2}\partial_{k}\phi+F\partial_{k}F)+2F(F\partial_{k}\phi+\partial_{k}F)_{;2}\}m_{j}\ell^{k}
  \nonumber \\
  & \quad +4F(F\partial_{k}\phi+\partial_{k}F)\ell_{j}\ell^{k} -2F(F\partial_{j}\phi+\partial_{j}F)].
\end{flalign}
From \eqref{g^ij.bar.2.dim}, \eqref{G_K Ceoff-1form} and \eqref{G_k coeff 2 form}, we get
\begin{flalign*}
\overline{G^{i}}  =&\frac{1}{4}\overline{g}^{ij}(y^{k}\,\dot{\partial}_{j}\partial_{k}\overline{F}^{2}-\partial_{j}\overline{F}^{2}) \\
    =& G^{i}+\frac{1}{4}[\{-2F(\phi_{;2})^{2}\rho(F\partial_{k}\phi+\partial_{k}F)-2F\phi_{;2}\rho(F\partial_{k}\phi+\partial_{k}F)_{;2}+2F^{2}\partial_{k}\phi\}\ell^{k}
  \\
 &+2\varepsilon F\rho \phi_{;2}(F\partial_{k}\phi+\partial_{k}F)m^{k}]\ell^{i}
+\frac{1}{4}[\{2\varepsilon F\phi_{;2}\rho(F\partial_{k}\phi+\partial_{k}F)+2\varepsilon F\rho(F\partial_{k}\phi+\partial_{k}F)_{;2}
  \\
 &-2F(\partial_{k}F)_{;2}\}\ell^{k}+\{-2F\rho(F\partial_{k}\phi+\partial_{k}F)+2\varepsilon F\partial_{k}F\}m^{k}]m^{i}.
\end{flalign*}
Hence,  the result follows.
\end{proof}

\begin{remark}
Using $\delta_k\phi=\partial_k \phi-G_k^i\dot{\partial}_i\phi$ and $\delta_kF^2=0$, we obtain the following technical equalities:
\begin{equation}
\renewcommand{\arraystretch}{1.5}
  \left.\begin{array}{r@{\;}l}
\emph{\textbf{(a)}}& F^2\ell^k\,\partial_k\phi=F^2\phi_{,1}+2G^k\phi_{;2}\;m_k,\\
\emph{\textbf{(b)}}& F m^k\,\partial_k\phi=\varepsilon F\,\phi_{,2}+G^i_{k}\;\phi_{;2}\;m^k m_i,\\
\emph{\textbf{(c)}}& \ell^k\,\partial_k F^2=4\,G^k\,\ell_{k},\\
\emph{\textbf{(d)}}& m^k\partial_k F^2=2FG_k^i\,\ell_{i}\,m^k,\\
\emph{\textbf{(e)}}& F^2\ell^k(\partial_k\phi)_{;2}=F^2\,\phi_{,1;2}+\varepsilon F G^k_i\,\phi_{;2}\;m_km^i+2G^k\;\phi_{;2;2}\;m_k+2\varepsilon\phi_{;2}\;\mathcal{I}\; G^k m_k\\
&\qquad\qquad\qquad-2G^k\,\phi_{;2}\,\ell_k-F^2\,\phi_{,2}\\
\emph{\textbf{(f)}}& \ell^k(\partial_k F^2)_{;2}=2\varepsilon F \,G^k_i\,\ell_k\, m^i+4\varepsilon\, G^k \,m_k.
  \end{array}\right\} \label{eq}
\end{equation}
Subsituting \eqref{eq} into the expressions of $P$ and $Q$ given by \eqref{formula for P, Q}, we obtain the following short and compact expressions:
\begin{align}
2Q&=\varepsilon\rho F^2(\phi_{;2}\,\phi_{,1}+\phi_{,1;2}-2\phi_{,2}),\label{formula of Q only}\\
2P&=-\rho F^2\phi_{;2}(\phi_{;2}\,\phi_{,1}+\phi_{,1;2}-2\phi_{,2})+F^2\phi_{,1}.\label{formula of P only}
\end{align}
Moreover, from \eqref{formula of Q only} and \eqref{formula of P only}, we get
\begin{align}\label{relation between P and Q}
2\varepsilon \,\phi_{;2}\,Q+2P=F^2\phi_{,1}.
\end{align}
\end{remark}
Hence, we get the following equivalent form of Proposition \ref{transformation of spray}.
\begin{theorem}
Under the anisotropic conformal transformation \eqref{the anisotropic conformal transformation2},   the geodesic spray coefficients $ \overline{G}^{i}$  and $G^{i}$ of $ \overline{F}$ and $F$, respectively,  are related by
\begin{align}\label{coefficet of spray 2}
\overline{G}^i=G^i+Qm^i+(\frac{1}{2} F^2\phi_{,1}-\varepsilon \phi_{;2}\;Q)\ell^i,
\end{align}
where $Q$ is given by \eqref{formula of Q only}.
\end{theorem}

\begin{remark}\label{Rem:spray of F bar}

\emph{\textbf{(a)}} As a direct consequence of  \eqref{coeff of sprayof F bar}, the geodesic spray $\overline{S}$ associated with $\overline{F}$ has the form
  \begin{equation*}\label{spray of F bar}
    \overline{S}=S-2(P\,\ell^{i}+Q\, m^{i})\,\dot{\partial_{i}}.
   \end{equation*}
   
\emph{\textbf{(b)}} From \eqref{coefficet of spray 2}, we get $G^i=\overline{G}^{i}$ if and only if $Q=0$ and $\phi_{,1}=0.$

   \end{remark}
  
\begin{proposition}\label{transform of Barthal }
Under the anisotropic conformal transformation \eqref{the anisotropic conformal transformation2}, the coefficients of the Barthel connections  $\overline{G}^{i}_{j}$  and $G^{i}_{j}$ associated with $\overline{F}$ and $F$, respectively, are related by
\begin{equation}\label{barthel coefficient}
F\overline{G}_{j}^{i} =FG_{j}^{i}+\left\lbrace 2P\ell^{i}\ell_{j}+(P_{;2}-Q)\ell^{i}m_{j}+2Q\ell_{j}m^{i}+(\varepsilon P+Q_{;2}-\varepsilon\mathcal{I}Q)m^{i}m_{j}\right\rbrace.
\end{equation}
\end{proposition}
\begin{proof}
The proof is obtained directly by differentiating \eqref{coeff of sprayof F bar}  with respect to $y^{j}$ and using  Lemma \ref{properties.2.dim.Fins}.
\end{proof}

\begin{proposition}\label{transform of berwald connection}
Under the anisotropic conformal transformation \eqref{the anisotropic conformal transformation2},   the coefficients of the Berwald connection  $\overline{G}^{i}_{jk}$  and $G^{i}_{jk}$ associated with $\overline{F}$ and $F$, respectively, are related by
\begin{align}\label{berwald connection}
 F^2\;\overline{G}^{i}_{jk} =&F^2\;{G}^{i}_{jk}+(2P\ell^{i}+2Qm^{i})\ell_{j}\ell_{k}+\{(P_{;2}-Q)\ell^{i}+(\varepsilon P+ Q_{;2}-\varepsilon\mathcal{I} Q)m^{i}\}(\ell_{j}m_{k}+\ell_{k}m_{j})\nonumber\\
 &+\{(\varepsilon P+P_{;2;2}-2Q_{;2}+\varepsilon\mathcal{I}P_{;2})\ell^{i}+(2\varepsilon P_{;2}
  +\varepsilon Q+ Q_{;2;2}-\varepsilon\mathcal{I}_{;2}Q-\varepsilon\mathcal{I} Q_{;2})m^{i}\}m_{j}m_{k}.
\end{align}
\end{proposition}
\begin{proof}
It follows from \eqref{barthel coefficient} by   differentiating $\overline{G}^i_j$ with respect to $y^{k}$ and using  Lemma \ref{properties.2.dim.Fins}. 
\end{proof}

 Geometric properties or geometric objects which are preserved under the anisotropic conformal transformation are said to be anisotropic conformal invariants. In general, the geodesic spray is not invariant under the anisotropic conformal change. 
 In the (isotropic) conformal transformation  $\overline{F}(x,y) = e^{\phi(x)} F(x,y)$, the homothetic transformation is the only case that leads to unchanged geodesic spray  \cite[Proposition 11.1]{Shenbook16}. For the anisotropic conformal transformation \eqref{the anisotropic conformal transformation2}, the geodesic spray is not invariant in general. Nevertheless, we have 
 
\begin{theorem}\label{phi dh-closed}
  Let $(M,F)$ be a conic pseudo-Finsler surface.  Under the anisotropic conformal transformation \eqref{the anisotropic conformal transformation2},  the following assertions are equivalent:
 \begin{description}
   \item[(i)] $\overline{S}=S$, that is, $P=Q=0,$ where $P$ and $Q$ are given by \eqref{formula of Q only} and \eqref{formula of P only}, respectively.
   \item[(ii)] $ \delta_{i}\phi=0$, or equivalently, $d_h\phi=0$.
 \end{description}
\end{theorem}

\begin{proof}
$\textbf{(i)}\Longrightarrow\textbf{(ii)}$:  $\overline{F}=e^{\phi}F$ implies
 \begin{equation}\label{delta F}
 \delta_{i}\overline{F}=\delta_{i}(e^{\phi}F)=e^{\phi}F\,\delta_{i}\phi +e^{\phi}\,\delta_{i}F=\overline{F}\delta_{i}\phi +e^{\phi}\,\delta_{i}F.
 \end{equation}
 It is known that $F$ is horizontally constant, that is,  $\delta_{i}F=0.$ By assumption $\overline{S}=S$ and so $\overline{\delta}_{i}=\delta_{i}$, then $$\delta_{i}\overline{F}=\overline{\delta}_{i}\overline{F}=0.$$
  Thereby, \eqref{delta F} gives $\overline{F}\delta_{i}\phi =0.$ Hence, $\delta_{i}\phi=0.$\vspace{6pt}\\
  $\textbf{(ii)}\Longrightarrow\textbf{(i)}$:
Suppose $\delta_{i}\phi=0$. Thus, Lemma \ref{properties.2.dim.Fins} $\textbf{(vii)}$ gives 
\begin{align}\label{delt_i=0}
\phi_{,1}\ell_{i}+\phi_{,2}m_{i}=0.
\end{align}
Contracting both sides of \eqref{delt_i=0} by $\ell^i$ and $m^i$, respectively, we obtain  $\phi_{,1}=\phi_{,2}=0$. Then, $P=0$ and $Q=0$  by \eqref{formula of Q only} and \eqref{formula of P only}. From \eqref{coeff of sprayof F bar},  $\overline{G}^i=G^i$ and so $\overline{S}=S$.  
\end{proof}

\begin{corollary}\label{invariants}
Let $(M,F)$ be a conic pseudo-Finsler surface.  Under the anisotropic conformal transformation \eqref{the anisotropic conformal transformation2},  if $\delta_{i}\phi=0$, then the Barthel connection $G^{i}_{j}$ and Berwald connection $G^{i}_{jk}$ are invariant.
\end{corollary}
 
\begin{proposition}
Let $(M,F)$ be a conic pseudo-Finsler surface.  Under the anisotropic conformal transformation \eqref{the anisotropic conformal transformation2}, the following assertions are equivalent:
\begin{description}
\item[(i)]$\overline{G}^{i}=G^{i}.$ \qquad \qquad\em{\textbf{(ii) }}$\overline{G}^{i}_{j}=G^{i}_{j}.$ \qquad\qquad \textbf{(iii)} $\overline{G}^{i}_{jk}=G^{i}_{jk}.$
\end{description}
\end{proposition}
\begin{proof}
\textbf{(i)}$\Longrightarrow$\textbf{(ii)} and \textbf{(ii)}$\Longrightarrow$\textbf{(iii)} follow directly by  differentiating $G^i$ with respect to $y^{j}$ and $G^i_j$ with respect to $y^{k},$ respectively.\\
\textbf{(iii)}$\Longrightarrow$\textbf{(i)}: From \eqref{berwald connection}, $\overline{G}^{i}_{jk}=G^{i}_{jk}$ implies that
\begin{align}\label{assertions of gedesic coeff, barthal berwald conn invariant}
0&=(2P\ell^{i}+2Qm^{i})\ell_{j}\ell_{k}+\{(P_{;2}-Q)\ell^{i}+(\varepsilon P+ Q_{;2}-\varepsilon\mathcal{I} Q)m^{i}\}(\ell_{j}m_{k}+\ell_{k}m_{j})+\{(\varepsilon P\nonumber\\
 &+P_{;2;2}-2Q_{;2}+\varepsilon\mathcal{I}P_{;2})\ell^{i}+(2\varepsilon P_{;2}
  +\varepsilon Q+ Q_{;2;2}-\varepsilon\mathcal{I}_{;2}Q-\varepsilon\mathcal{I} Q_{;2})m^{i}\}m_{j}m_{k}.
\end{align}
Contracting \eqref{assertions of gedesic coeff, barthal berwald conn invariant} by $\ell_{i}\ell^{j}\ell^{k}$ and $m_{i}\ell^{j}\ell^{k},$ respectively, we get $Q=P=0$.  Hence, by Theorem \ref{phi dh-closed},  the coefficients of the geodesic spray  are invariant. 
  \end{proof}

The following example provides a non-trivial anisotropic conformal transformation which leaves the geodesic spray invariant.  A Maple's code of the detailed calculations is found in the Appendix at the end of the paper.

\begin{example}\label{example 2}
Let $M=\mathbb{B}^2\subset\mathbb{R}^2$, $x \in M,\; y\in T_{x}\mathbb{B}^{2}\cong \mathbb{R}^{2},\; a=(a_{1},a_{2})\in\mathbb{R}^{2} $ and $a$ is constant vector with $|a|<1$. Let
$$\displaystyle{z^i=\frac{(1+\langle a,x\rangle)y^i-\langle a,y \rangle x^i}{\langle a,y \rangle}}.$$
Define the Finsler metric $F$ by
$$F=\frac{\langle a,y\rangle \sqrt{(z^1)^2+(z^2)^2}}{(1+\langle a,x\rangle)^2}. $$
The  geodesic spray coefficients are given by
$$
G^i=-\frac{\langle a,y \rangle}{1+\langle a,x \rangle}y^i.
$$
Now, let\, $\overline{ F}= e^{\phi} F=\dfrac{\langle a,y\rangle \sqrt{(z^1)^2+(z^2)^2}}{(1+\langle a,x\rangle)^2} \ \exp\left( \sqrt{(z^1)^2+(z^2)^2}\right), $
where $\phi=\sqrt{(z^1)^2+(z^2)^2}.$ \\
One can easily check that $\overline{ F}$ satisfies Theorem \emph{\ref{ness.suff.codtion.F.bar.Finsler}} (which means that $\overline{F}$ is a pseudo-Finsler metric) and $\delta_{i}\phi=0$. Furthermore,
$$\overline{G}^i=G^i=-\frac{\langle a,y \rangle}{1+\langle a,x \rangle}\,y^i.$$
\end{example}

%%%%%%%%%%%%%%%%%%%%%%%%%%%%%%% Sec. 4 %%%%%%%%%%%%%%%%%%%%%%%%%%%%%%%%
\section{Anisotropic conformal and Projective changes}
\begin{definition}\emph{\cite{gendi}}
Two sprays $S$ and $\hat{S}$on  $M$ are projectively  equivalent (or related) if there exists an $h(1)$-function $\mathcal{P}: TM \longrightarrow \mathbb{R}$ such that
  $\hat{S}=S-2\mathcal{P} \,\mathcal{C}.$ The function $\mathcal{P}$ is called the projective factor.  Locally, $\hat{G}^{i}=G^{i}+\mathcal{P} y^{i}.$
\end{definition}
The association $S\longmapsto\hat{S}=S-2\mathcal{P} \,\mathcal{C}$ or $G\longmapsto\hat{G}^{i}=G^{i}+\mathcal{P} y^{i}$ is said to be projective transformation (or projective change).

\begin{theorem}\label{projectively equivalent}
  Under the anisotropic conformal transformation \eqref{the anisotropic conformal transformation2}, the following assertions are equivalent:
\begin{description}
\item[(i)]the geodesic sprays $S$ and $\overline{S}$ are projectively equivalent,
\item[(ii)]$\phi_{;2}\phi_{,1}+\phi_{,1;2}=2\phi_{,2},$
\item[(iii)]$\overline{G}^{i}=G^{i}+\frac{1}{2} F\phi_{,1}y^i.$
\end{description}  
\end{theorem}
\begin{proof}
 From \eqref{the anisotropic conformal transformation2} and \eqref{coeff of sprayof F bar}, the two geodesic sprays $S$ and $\bar{S}$ are projectively equivalent if and only if $\overline{G}^{i}=G^{i}+\mathcal{P} y^{i} $ where  $\mathcal{P}=\dfrac{P}{F}$ and $Q=0$. This is equivalent to
  $\phi_{;2}\phi_{,1}+\phi_{,1;2}-2\phi_{,2}=0,$ by  \eqref{formula of Q only}. 
In this case \eqref{formula of P only} takes the form
$2P=F^2\phi_{,1}$ and hence  the geodesic spray of $\overline{F}$ is given by
\begin{equation*}
\overline{G}^{i}=G^{i}+\frac{1}{2} F\phi_{,1}\;y^i.
\end{equation*}
\vspace*{-1.1cm}\[\qedhere\]
\end{proof}

In view of Theorem \ref{projectively equivalent}, we have
\begin{corollary}
Under the anisotropic conformal transformation \eqref{the anisotropic conformal transformation2} with $\phi_{,1}=0$, the coefficients of the geodesic spray are invariant if and only if the two geodesic sprays $S$ and $\bar{S}$ are projectively equivalent. 
\end{corollary}

\vspace*{.4cm}

\begin{definition}\emph{\cite{chernbook2006}}
A conic pseudo-Finsler metric $F=F(x,y) $  on an open subset $U \subset \mathbb{R}^{n}$ is said to be projectively flat  if any one of the following equivalent conditions is satisfied:
\begin{description}
\item[(i)]the geodesics of F  are straight line segments in $U$,
\item[(ii)]$y^{j}\partial_{j}\dot{\partial}_{i} F=\partial_{i}F$ or equivalently $ \dot{\partial}_{i}(y^{j}\partial_{j} F)=2\partial_{i}F,$
\item[(iii)]the geodesic spray coefficients $G^i$ are of the  form $G^{i}=\mathcal{P} y^i$, with the  projective factor $\mathcal{P}$ is given by $\mathcal{P}=\dfrac{y^k\partial_{k}F}{2F}$.
\end{description}  
\end{definition}

 \begin{theorem}\label{projetively flat}
Under the anisotropic conformal transformation \eqref{the anisotropic conformal transformation2}, we have the following:
\begin{description}
\item[(i)]a necessary condition for $\overline{F}$ to be  projectively flat is that\, $Q+\varepsilon G^k m_k=0.$
\item[(ii)]a sufficient condition for $\overline{F}$ to be  projectively flat is that\, $F \partial_{j}\phi+\partial_{j}F=0$. 
\end{description} 
\end{theorem}
\begin{proof}

Since $\overline{F}=e^{\phi}F$, we get
\begin{equation}\label{F bar deriv w.r.t x^i}
\partial_{j}\overline{F}=e^{\phi}[F \partial_{j}\phi+\partial_{j}F]
\end{equation}
  and\, 
$y^{j}\partial_{j}\overline{F}=e^{\phi}[F^{2} (\partial_{j}\phi)\ell^{j}+F(\partial_{j}F)\ell^{j}].$
Thereby,
\begin{align}\label{F bar deriv w.r.t x^i and y}
\dot{\partial}_{i}(y^{j}\partial_{j} \overline{F})&=e^{\phi}[\{\phi_{;2}F \partial_{j}\phi+\phi_{;2}\partial_{j}F+F (\partial_{j}\phi)_{;2}+(\partial_{j}F)_{;2}\}\ell^{j}m_{i}+\{2F \partial_{j}\phi+2\partial_{j}F\}\ell_{i}\ell^{j}\nonumber\\
&\;\;\;+\{\varepsilon F \partial_{j}\phi+\varepsilon \partial_{j}F\}m^{j}m_{i}].
\end{align}
From \eqref{F bar deriv w.r.t x^i} and \eqref{F bar deriv w.r.t x^i and y}, we get
\begin{align}\label{F bar proj flat nessc}
\dot{\partial}_{i}(y^{j}\partial_{j} \overline{F})-2\partial_{i}\overline{F}=&e^{\phi}[\{\phi_{;2}(F \partial_{j}\phi+\partial_{j}F)+(F \partial_{j}\phi+\partial_{j}F)_{;2}\}\ell^{j}m_{i}+2(F \partial_{j}\phi+\partial_{j}F)\ell_{i}\ell^{j}\nonumber\\
&+\varepsilon (F \partial_{j}\phi+\partial_{j}F)m^{j}m_{i}-2(F\partial_{i}\phi+\partial_{i}F)].
\end{align}

\textbf{(i)} Let $\overline{F}$ be projectively flat, i.e., $\dot{\partial}_{i}(y^{j}\partial_{j} \overline{F})-2\partial_{i}\overline{F}=0$. 
Then, the RHS of \eqref{F bar proj flat nessc} vanishes. Contracing the later by $m^i$ and using \eqref{eq}, we obtain $$\frac{2}{F\rho}(Q+\varepsilon G^k m_k)=0,$$ which implies that $Q+\varepsilon G^k m_k=0.$

\textbf{(ii)} The proof is obtained directly from \eqref{F bar proj flat nessc}.

\vspace*{-1.1cm}\[\qedhere\] 
\end{proof}

\begin{remark}\label{remark equiv of F proje flat}
\emph{\textbf{(a)}} \emph{\cite[\S 3.1]{Matsumoto 2003}} Since  $0=\delta_i F=\partial_i F-G^r_i\ell_r$, we get
\begin{equation}\label{Eq:G^i ell_i}
2G^i\ell_i=y^i\partial_iF \text{ and }\,  2G^r m_r=\frac{\varepsilon F^2 \mathfrak{M}}{\mathrm{h}},
\end{equation}
where $\mathrm{h}:=\sqrt{\varepsilon\mathfrak{g}}$ and $\mathfrak{M}:=\dot{\partial}_{2}\partial_1F-\dot{\partial}_{1}\partial_2F$.
From \eqref{Eq:G^i ell_i}, we get
\begin{equation}\label{Eq:of coeff spray 2 dim}
 2G^i =y^r(\partial_r F)\ell^i+\frac{ F^2 \mathfrak{M}}{\mathrm{h}}m^i. 
\end{equation}

\emph{\textbf{(b)}} From \emph{\textbf{(a)}}, the conic pseudo-Finsler surface is projectively flat if and only if $G^i m_i=0$, which is equivalent to $\mathfrak{M}=\dot{\partial}_{2}\partial_1F-\dot{\partial}_{1}\partial_2F=0$ (Hammel equation in two-dimensional spaces) and the geodesic spray coefficients have the form $G^i =\frac{y^r(\partial_r F)}{2F}y^i.$ 
      
\end{remark}

\begin{theorem}\label{preserve flatenss}
Let $F $ be a projectively flat  conic pseudo-Finsler metric and let $\overline{F}=e^{\phi}F$ be the anisotropic conformal transformation  \eqref{the anisotropic conformal transformation2}. Then, the Finsler metric  $\overline{F}$ is projectively flat if and only if  $\phi_{;2}\phi_{,1}+\phi_{,1;2}-2\phi_{,2}=0.$ 
\end{theorem}
\begin{proof}
By Theorem \ref{projectively equivalent},  we have $\phi_{;2}\phi_{,1}+\phi_{,1;2}-2\phi_{,2}=0$  is equivalent to $\overline{G}^{i}=G^{i}+\frac{1}{2} F\phi_{,1}y^i$. Since $F$ is projectively flat, then we get
\begin{align}\label{F bar projectively flat 1}
\overline{G}^{i}=\dfrac{y^k\partial_{k}F+F^2\phi_{,1}}{2F} \,y^i.
\end{align}
 By using $\text{(a)}$  of \eqref{eq}, we obtain
\begin{equation}\label{projective factor}
\frac{y^k\partial_{k}\overline{F}}{2\overline{F}} =\frac{y^k\partial_{k}F+y^kF \partial_{k}\phi}{2F}=\frac{y^k\partial_{k}F+F^2\phi_{,1}+2G^k\phi_{;2}\;m_k}{2F}=\frac{y^k\partial_{k}F+F^2\;\phi_{,1}}{2F},
\end{equation}
since $G^k\;m_k=0$, by Remark \ref{remark equiv of F proje flat}.  
  From \eqref{F bar projectively flat 1} and \eqref{projective factor}, we get
$\overline{G}^{i} =\overline{\mathcal{P}}\,y^i$,
with \,\,  $\overline{\mathcal{P}}=\dfrac{y^k\partial_{k}\overline{F}}{2\overline{F}}$.\\ This is equivalent to the projective flatness of the metric $\overline{F}$.
\end{proof}

\begin{proposition}
Let $(M,F)$ be a conic pseudo-Finsler surface and let $\overline{F}=e^{\phi}F$ be the anisotropic conformal transformation  \eqref{the anisotropic conformal transformation2} with $\phi_{;2}\phi_{,1}+\phi_{,1;2}-2\phi_{,2}=0.$ Assume that $\overline{F} $ is projectively flat. Then, the Finsler metric $F$ is projectively flat if and only if either $G^km_k=0$ or the conformal factor $\phi$ is a function of $x$ only.   
\end{proposition}
\begin{proof}
From Theorem \ref{projectively equivalent},  we have $\phi_{;2}\phi_{,1}+\phi_{,1;2}-2\phi_{,2}=0$  is equivalent to $\overline{G}^{i}=G^{i}+\frac{1}{2} F\phi_{,1}y^i$. Since $\overline{F}$ is projectively flat,  we obtain
\begin{align*}
G^{i}=\overline{G}^{i}-\frac{1}{2} F\phi_{,1}\;y^i=\frac{y^k\partial_{k}F+y^k F\partial_{k}\phi-F^2\;\phi_{,1}}{2F} \;y^i.
\end{align*}
 By using $\text{(a)}$  of \eqref{eq}, we get
 \begin{align}\label{F projectivelly flat}
G^{i}=\frac{y^k\partial_{k}F+2G^k\; \phi_{;2}\;m_k}{2F} \;y^i.
\end{align}
From \eqref{F projectivelly flat}, $F$ is projectively flat if and only if $G^k \phi_{;2}m_k=0,$ which is equivalent to  either $G^km_k=0$ or the conformal factor $\phi$ is a function of $x$ only.  
\end{proof}

\begin{definition}\emph{\cite{chengshenzhou-on class f dually}}
A conic pseudo-Finsler metric $F=F(x,y) $  on an open subset $U \subset \mathbb{R}^{n}$ is said to be dually flat  if it satisfies the following equations:
$$y^{j}\dot{\partial}_{i}\partial_{j} F^2=2\partial_{i}F^2.$$ 
\end{definition}

\begin{theorem}\label{dually flat}
Under the anisotropic conformal transformation \eqref{the anisotropic conformal transformation2}, the following hold:
\begin{description}
\item[(i)]a necessary condition for $\overline{F}$ to be dually flat is that $$\frac{2}{F\rho}( Q+\varepsilon  G^k m_k)+\frac{2\varepsilon\phi_{;2}}{F}G^k m_k-G^i_k\ell_i m^k-\phi_{;2}G^i_k m_im^k+\varepsilon F(\phi_{;2}\phi_{,1}-\phi_{,2})=0.$$
\item[(ii)] a sufficient condition for $\overline{F}$ to be dually flat is that $F \partial_{j}\phi+\partial_{j}F=0$. 
\end{description} 
\end{theorem}
\begin{proof}
As $\overline{F}^{2}=e^{2\phi}F^{2}$, then $\partial_{j}\overline{F}^{2}=2e^{2\phi}[F^{2} \partial_{j}\phi+F\partial_{j}F]$.  Consequently, one can show that 
\begin{align}\label{F dually proj flat}
y^{j}\dot{\partial}_{i}(\partial_{j} \overline{F}^{2})-2\partial_{i}\overline{F}^{2}=&2Fe^{2\phi}[\{2\phi_{;2}(F \partial_{j}\phi+\partial_{j}F)+(F \partial_{j}\phi+\partial_{j}F)_{;2}\}\ell^{j}m_{i}+2(F \partial_{j}\phi+\partial_{j}F)\ell_{i}\ell^{j}\nonumber\\
&-2(F \partial_{i}\phi+\partial_{i}F)].
\end{align}
\textbf{(i)} Let $\overline{F}$ be dually flat, that is, $y^{j}\dot{\partial}_{i}(\partial_{j} \overline{F}^{2})-2\partial_{i}\overline{F}^{2}=0$. Then, the RHS of \eqref{F dually proj flat} vanishes. Contracting the later by $m^i$ and using \eqref{eq} and \eqref{formula of Q only}, we get
$$\frac{2}{F\rho}( Q+\varepsilon  G^k m_k)+\frac{2\varepsilon\phi_{;2}}{F}G^k m_k-G^i_k\ell_i m^k-\phi_{;2}G^i_k m_im^k+\varepsilon F(\phi_{;2}\phi_{,1}-\phi_{,2})=0.$$
\textbf{(ii)} The proof is obtained directly From \eqref{F dually proj flat}. 

\vspace*{-1.1cm}\[\qedhere\]
\end{proof}

The following example shows that the condition $F \partial_{j}\phi+\partial_{j}F=0$ is not a necessary condition for neither projective flatness nor dual flatness of $\overline{F}$ (A Maple's code of the detailed calculations is found in the Appendix). On the other hand, the same condition $F \partial_{j}\phi+\partial_{j}F=0$ is a sufficient condition for both projective flatness and dual flatness of $\overline{F}$

\begin{example}\label{example of non necessary condition of dually flat}
Let $F$ be the Klein metric on the unit ball  $B^n\subset\mathbb{R}^{n},$
$F=\dfrac{\sqrt{|y|^2-(|x|^2|y|^2-\langle x, y\rangle^2)}}{1-|x|^2}$. We get
$$\overline{ F}= e^{\phi} F=\frac{\sqrt{|y|^2-(|x|^2|y|^2-\langle x, y\rangle^2)}}{1-|x|^2}+\frac{\langle x, y\rangle}{1-|x|^2}, $$
where $$\phi=\ln \left(1+\frac{\langle x, y\rangle}{\sqrt{|y|^2-\left(|x|^2|y|^2-\langle x, y\rangle^2\right)}} \right).$$
$\overline{F}$ is both projectively flat and dually flat \emph{\cite{chengshenzhou-on class f dually}}, but $F \partial_{j}\phi+\partial_{j}F\neq0.$
\end{example}

%%%%%%%%%%%%%%%%%%%%%%%%%%%% Sec. 5 %%%%%%%%%%%%%%%%%%%%%%%%%%%%%%%%%

\section{Special cases for the anisotropic conformal factor  }\label{section five}

Based on the importance of the conformal factor $\phi(x,y)$ in studying the geometric objects associated with $\overline{F}$,  we  focus our attention on studying some special cases of $\phi$. We consider two cases: the function $\phi$ is a function of $y$ only and the function $\phi$ is a function of $x$ only. In the later case the  anisotropic conformal transformation reduces to the (isotropic) conformal~transformation~\cite{Hashiguchi76}.

\begin{proposition}
  Let $(M,F)$ be a conic pseudo-finsler surface and consider the anisotropic conformal transformation \eqref{the anisotropic conformal transformation2} with $\phi$ a function of $y $ only. Assume that $F$ is projectively flat. Then, $\overline{F}$ is projectively flat if and only if $\phi_{,2}=0.$
\end{proposition}
\begin{proof}
  Since $\phi$ is a function of $y $ only, we get
\begin{equation}\label{equ:for y 1}
\delta_i\phi=\phi_{,1}\ell_{i}+\phi_{,2}m_i=-\frac{\phi_{;2}}{F}G^r_im_r.
\end{equation}
Multiple both sides of \eqref{equ:for y 1} by $\ell^i$, we get
\begin{equation}\label{equ:for y 2}
F^2\phi_{,1}=-2\phi_{;2}G^rm_r.
\end{equation}
As $F$ is projectively flat, from Remark \ref{remark equiv of F proje flat} and \eqref{equ:for y 2}, we have 
\begin{equation}\label{equ:for y 3}
\phi_{,1}=0.
\end{equation}
By Theorem \ref{preserve flatenss} and \eqref{equ:for y 3}, $\overline{F}$ is projectively flat  if and only if $\phi_{,2}=0.$   
  \end{proof}

\begin{proposition}
Let $(M,F)$ be a conic pseudo-finsler surface and consider the anisotropic conformal transformation \eqref{the anisotropic conformal transformation2} with $\phi$ a function of $y $ only. Assume that $\overline{F}$ is projectively flat. Then, a necessary condition for $\overline{F}$ to be dually flat is that 
$$\varepsilon F(\phi_{;2}\phi_{,1}-\phi_{,1})-G^i_k\ell_i m^k=0.$$
\end{proposition}
\begin{proof}
Since $\overline{F}$ is projectively flat, then by Theorem \ref{projetively flat}, we have 
\begin{equation}\label{first formula phi of y}
Q+\varepsilon G^km_k=0.
\end{equation} 
Also, $\phi$ is  function of $y$ only, from $(a)$ and $(b)$ of \eqref{eq}, we get
\begin{equation}\label{second formula phi of y}
\frac{2\varepsilon\phi_{;2}}{F}G^k m_k-\phi_{;2}G^i_k m_im^k+\varepsilon F(\phi_{,1}-\phi_{,2})=0.
\end{equation}
From \eqref{first formula phi of y}, \eqref{second formula phi of y} and Theorem \ref{dually flat} if  $\overline{F}$ is  dually flat, then $\varepsilon F(\phi_{;2}\phi_{,1}-\phi_{,1})-G^i_k\ell_i m^k=0.$    
\end{proof}

\begin{lemma}\label{phi of x only}
  Let $(M, F)$  be a conic pseudo-Finsler surface and $\overline{F}=e^{\phi }F$ be the anisotropic conformal transformation \eqref{the anisotropic conformal transformation2}. If $\phi$ is function of $x$ only, then
\begin{equation}\label{condition rho varepslon}
    \sigma=0,    \quad     \rho=\varepsilon.
  \end{equation}
\end{lemma}
\begin{proof}
  Since $\phi$ is independent of $y$, then $\phi_{;2}=0$ and $ \sigma=\phi_{;2;2}+\varepsilon \mathcal{I}\phi_{;2}+2(\phi_{;2})^{2}=0.$  Also we have $ \rho=\dfrac{1}{\sigma+\varepsilon-(\phi_{;2})^{2}}=\dfrac{1}{\varepsilon}=\varepsilon.$
\end{proof}

\begin{lemma}
 Let $(M,F)$ be a conic pseudo-Finsler surface and $\overline{F}(x,y)=e^{\phi(x,y)}F(x,y).$ Then, $\overline{g}_{ij}  =  e^{2\phi} g_{ij}$ if and only if $\phi$ is either a function of $x$ only or is  a constant (homothetic).
\end{lemma}
\begin{proof}
Let $\phi$ be a function of $x$ only and $\overline{F}=e^{\phi}F$. It is obvious that $\overline{g}_{ij}  =  e^{2\phi} g_{ij}$ (with no further condition). Conversely, let  $\overline{g}_{ij}  =  e^{2\phi} g_{ij}.$ Then, by \eqref{2.form.g_ij.bar.2.dim}, we get $$e^{2\phi}[\phi_{;2}(\ell_{i}m_{j}+\ell_{j}m_{i})+\sigma m_{i}m_{j}]=0. $$
 Contracting  both sides of the previous equation by $\ell^{i}m^{j}$, we obtain
$\phi_{;2}=0. $
Hence, by \eqref{phi_i.2.dim}, $\phi$ is either a function of $x$ only or is  a constant.
\end{proof}
\begin{proposition}\label{transform for conformal}
  Let $\overline{F}=e^{\phi} F$ be an anisotropic conformal transformation. If $\phi $ is a function of $x $ only, then
 \begin{description}
   \item[(i)] $\overline{g}_{ij}=e^{2\phi}g_{ij},\qquad\overline{g}^{ij}=e^{-2\phi}g^{ij},$
   \item[(ii)] $\overline{\ell}^i=e^{-\phi}\ell^{i},\qquad\overline{\ell}_{i}=e^{\phi}\ell_{i},$
   \item[(iii)]$\overline{m}_{i}=e^{\phi}m_{i},$ \qquad $\overline{m}^{i}=e^{-\phi}m^{i},$\qquad  $\overline{h} _{ij}=e^{2\phi}h_{ij},$
   \item[(iv)] $\overline{C} _{ijk}=e^{2\phi}C_{ijk},\qquad \overline{C}^{i} _{jk}=C^{i}_{jk}$,\qquad $\overline{\mathcal{I}}=\mathcal{I},$
   \item[(v)]$ \overline{G}^{i}=  G^{i}+\frac{1}{2}F^{2}(\phi_{,1}\ell^{i}-\phi_{,2}m^{i}).$
 \end{description}
\end{proposition}
\begin{proof} The proof of
\textbf{(i)}-\textbf{(iv)} follow  by substituting  \eqref{condition rho varepslon}  into   Propositions \ref{tras-of.l_i.l^j.c_ijk} and   \ref{transform of cartan}.\\
   \textbf{(v)} If $\phi$ is a function of $x$ only, then by \eqref{formula for P, Q}, we have
   \begin{equation*}
   P=\frac{1}{2}F^{2}\ell^{i}\partial_{i}\phi, \qquad Q=-\frac{1}{2}\varepsilon F^{2}m^{i}\partial_{i}\phi.
   \end{equation*}
By  $(1)$ and $(2)$ of \eqref{eq} and \eqref{coeff of sprayof F bar}, where $\phi_{;2}=0$, we get
\begin{equation*}\label{p,Q phi of x only 2}
   \overline{G}^{i}=  G^{i}+\frac{1}{2}F^{2}(\phi_{,1}\ell^{i}-\phi_{,2}m^{i}).
   \end{equation*}
   \vspace*{-1.3cm}\[\qedhere\]
\end{proof}
\begin{remark}
In  view of  Proposition \emph{\ref{transform for conformal}}, under the anisotropic conformal transformation with $\phi$ is a function of $x$ only,  we conclude the following
\begin{description}
\item[(a)] $\overline{S}$ and $ S$ coincide if and only if  $\phi$ is homothetic.
\item[(b)]From \emph{\textbf(v)} and \eqref{Eq:of coeff spray 2 dim}, we get 
\begin{equation*}
\overline{G}^{i}=  \frac{1}{2}F[(\ell^r\partial_{r}F+F\phi_{,1})\ell^{i}+F(\frac{\mathfrak{M}}{h}-\phi_{,2})m^{i}].
\end{equation*} 
\item[(c)]From \emph{\textbf{(b)}} the geodesic spray of $\overline{F}$ is flat if and only if $\ell^r\partial_{r}F+F\phi_{,1}=0$ and $\frac{\mathfrak{M}}{h}=\phi_{,2}.$
\end{description}
 \end{remark} 

\section{Concluding remarks}
We end the present paper with the following comments and remarks.
In this paper, we have  investigated the notion of anisotropic conformal transformation of conic pseudo-Finsler metrics in two-dimensional manifolds using the associated modified Berwald frame. The following points are to be singled out: 
\begin{itemize}

\item The anisotropic conformal change of a pseudo-Finsler metric $F(x, y)$ does not necessarily yield a pseudo-Finsler metric $\overline{F}(x,y)$.  Consequently, we find out the necessary and sufficient condition for $\overline{F}(x,y) = e^{\phi(x,y)}F(x, y)$ to be a pseudo-Finsler metric. This is a crucial result which ensures that two (conic) pseudo-Finsler surfaces may or may not be anisotropically conformally related, contrary to the \lq\lq pure" Finslerian case where any two Finsler metrics are anisotropically conformally related ( $\overline{F}=(\frac{\overline{F}}{F})F$ ). This justifies our choice of the object of study and certifies that our study is nontrivial.

\item We study some geometric objects associated with the transformed metric such as, the Berwald frame, metric tensor, Cartan tensor, inverse metric tensor, main scalar, geodesic spray, Barthel connection and Berwald connection.

\item It has been proved that the geodesic spray is an anisotropic conformal invariant if and only if the conformal factor $\phi(x,y)$ has the property that $\delta_i\phi=0$ or equivalently is $d_h$-closed. On the other hand, it is well-known that in the (isotropic) conformal case, the geodesic spray is an (isotropic) conformal invariant if and only if the conformal factor $\phi(x)$ is homothetic \cite{Bacso Sandor-Cheng, Shenbook16}. As the condition of being $d_h$-closed is more general than the condition of being homothetic, our result includes as a special case the (isotropic) conformal result. On the other hand, as a consequence of our result, the anisotropic conformal invariance of any one of the geodesic spray, Cartan nonlinear connection or Berwald connection implies the aniotropic conformal invariance of the other ones.

\item Under the anisotropic conformal transformation $\overline{F}= e^{\phi}F$, for two projectively related conic pseudo-Finsler surfaces, the pojective flatness property is preserved. Moreover, the condition $F\partial_j\phi+\partial_jF=0$ is a sufficient condition for both projective flatness and dual flatness of $\overline{F}$, but is neither a necessary condition for projective flatness nor dual flatness of $\overline{F}$.

\item Two interesting special cases are considered: (i) when the anisotropic conformal factor depends on position only, (ii) when the anisotropic conformal factor depends on direction only. In the first case our anisotropic conformal transformation reduces to the well-known (isotropic) conformal transformation which was initiated in \cite{Hashiguchi76} and in the second case some interesting results are obtained.
\end{itemize}

\emph{This work will be continued in a forthcoming paper:  \quotes{Anisotropic conformal transformation of conic pseudo-Finsler surface, $II$}, where further geometric properties along with special Finsler spaces will be investigated}.

\includepdf[pages=1-4]{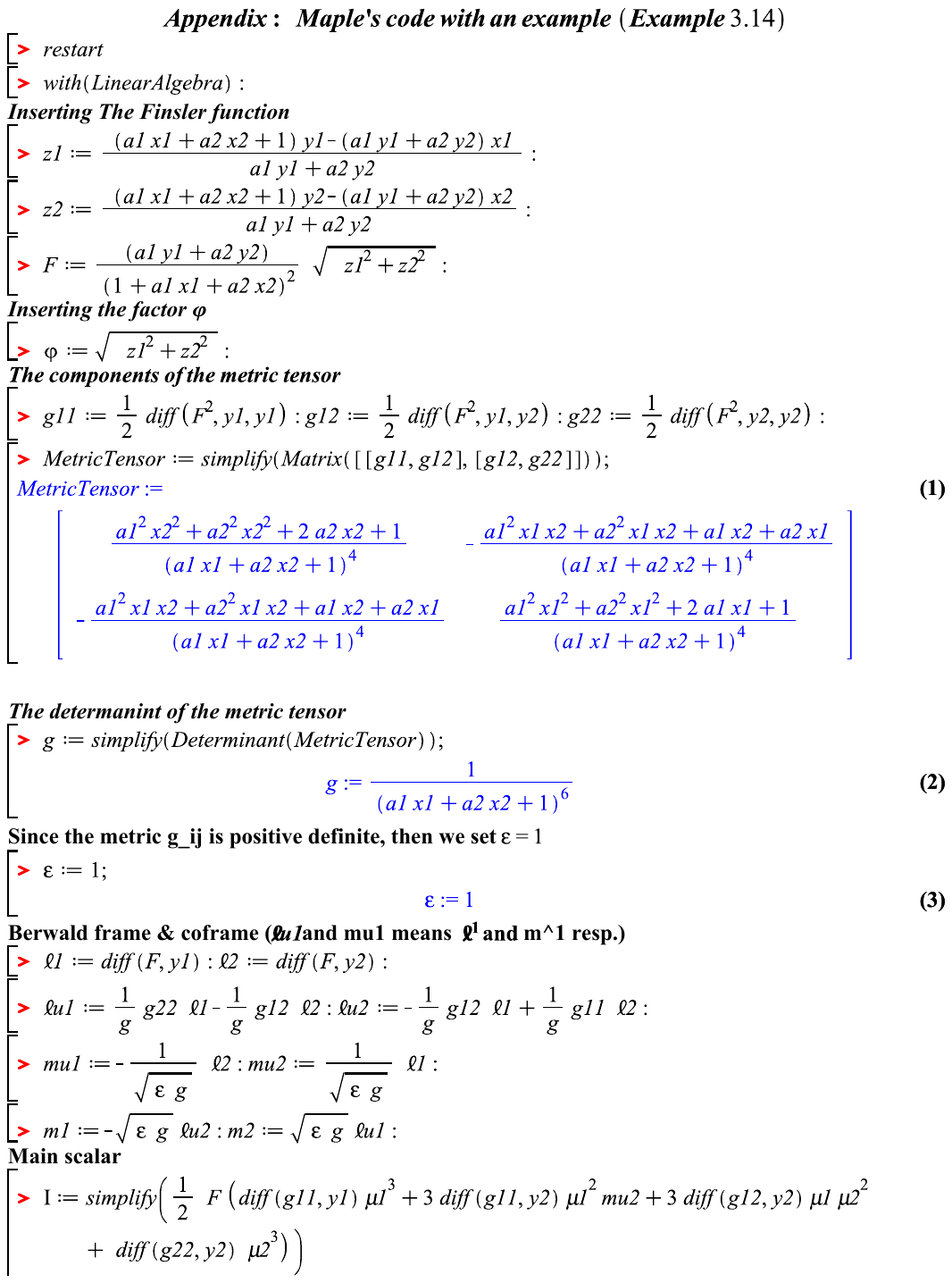}

\end{document}